\documentclass{article}

\usepackage{amsmath,amssymb,amsthm}
\usepackage{mathtools}
\usepackage{mathrsfs}
\usepackage{tikz-cd}
\usepackage{pgf, tikz}
\usetikzlibrary{arrows, shapes, decorations, matrix, positioning}
\usepackage{tabularx}
\usepackage{todonotes}
\usepackage{eucal}

\usepackage{fancyvrb}
\usepackage{fvextra}

\usepackage{bbding}
\newtheorem{thm}{Theorem}[section]
\newtheorem{cor}[thm]{Corollary}
\newtheorem{prop}[thm]{Proposition}
\newtheorem{lem}[thm]{Lemma}
\newtheorem{conj}[thm]{Conjecture}
\newtheorem{ques}[thm]{Question}

\newtheorem*{thm*}{Theorem}
\newtheorem*{cor*}{Corollary}
\newtheorem{ansatz}[thm]{Ansatz}
\theoremstyle{definition}
\newtheorem{defn}[thm]{Definition}
\newtheorem*{defn*}{Definition}
\newtheorem{rmk}[thm]{Remark}
\newtheorem{exam}[thm]{Example}

\usepackage{microtype}

\usepackage[all]{xy}
\usepackage{tikz-cd}

\usepackage{todonotes}

\usepackage{hyperref}
 \hypersetup{
   colorlinks   = true, 
   urlcolor     = blue, 
   linkcolor    = blue, 
   citecolor   = red 
 }
    
\newcommand{\co}{\colon\thinspace}
\newcommand{\mb}[1]{\mathbb{#1}}

\newcommand{\EE}{\mathbb{E}}

\newcommand{\sph}{\mathbb{S}}
\newcommand{\ZZ}{\mb Z}
\newcommand{\CC}{\mathbb{C}}
\newcommand{\RR}{\mathbb{R}}
\newcommand{\FF}{\mathbb{F}}

\newcommand{\scrF}{\mathscr{G}_{\FF}}

\newcommand{\modmod}{/\!\!/}

\DeclareMathOperator{\GL}{GL}
\DeclareMathOperator{\BGL}{BGL}
\DeclareMathOperator{\bgl}{bgl}
\DeclareMathOperator{\Ext}{Ext}
\DeclareMathOperator{\Hom}{Hom}
\DeclareMathOperator{\Map}{Map}

\DeclareMathOperator{\Sq}{Sq}
\DeclareMathOperator{\Pic}{Pic}
\DeclareMathOperator{\Br}{Br}
\DeclareMathOperator{\Sp}{\mathcal{S}p}
\DeclareMathOperator{\Spaces}{\mathcal{S}}
\DeclareMathOperator{\Top}{Top}

\DeclareMathOperator{\Cat}{Cat}

\DeclareMathOperator{\LMod}{LMod}

\DeclareMathOperator{\pic}{pic}
\DeclareMathOperator{\br}{br}
\DeclareMathOperator{\gl}{gl}
\DeclareMathOperator{\Gpd}{Gpd}
\DeclareMathOperator{\Mfd}{Mfd}

\newcommand{\KO}{KO}
\newcommand{\ko}{ko}
\newcommand{\ku}{ku}
\newcommand{\KU}{KU}

\DeclareMathOperator*{\fib}{fib}
\DeclareMathOperator*{\cofib}{cofib}
\DeclareMathOperator{\im}{im}

\newcommand{\LR}{L_{\mathbb{R}}}

\DeclareMathOperator{\FHT}{\mathscr{L}^{\mathbb{F}}}
\DeclareMathOperator{\FHTC}{\mathscr{L}^{\mathbb{C}}}
\DeclareMathOperator{\FHTR}{\mathscr{L}^{\mathbb{R}}}

\newcommand{\String}{\mathbf{String}}
\newcommand{\Spin}{\mathbf{Spin}}
\newcommand{\BString}{\mathbf{BString}}
\newcommand{\bstring}{\mathbf{bstring}}
\newcommand{\BSpin}{\mathbf{BSpin}}
\newcommand{\BSO}{\mathbf{BSO}}
\newcommand{\bso}{\mathbf{bso}}
\newcommand{\SpecialO}{\mathbf{SO}}
\newcommand{\BO}{\mathbf{BO}}

\newcommand{\UU}{\mathbf{U}}
\newcommand{\OO}{\mathbf{O}}

\newcommand{\CG}{DSV_{\mb F}}
\newcommand{\CGC}{DSV_{\mb C}}
\newcommand{\CGR}{DSV_{\mb R}}

\newcommand{\GBC}{\Pic_0^3(\KU)}
\newcommand{\GBCconn}{\Pic_1^3(\KU)}
\newcommand{\gbc}{\pic_0^3(\KU)}
\newcommand{\gbcconn}{\pic_1^3(\KU)}
\newcommand{\GBR}{\Pic_0^2(\KO)}
\newcommand{\GBRconn}{\Pic_1^2(\KO)}
\newcommand{\gbr}{\pic_0^2(\KO)}
\newcommand{\gbrconn}{\pic_1^2(\KO)}

\title{Brauer-Wall Groups and Truncated Picard Spectra of $K$-theory}
\author{Jonathan Beardsley, Kiran Luecke and Jack Morava}

\begin{document}
\maketitle

\begin{abstract}
    We compute the first two $k$-invariants of the Picard spectra of $KU$ and $KO$ by analyzing their Picard groupoids and constructing their unit spectra as global sections of sheaves on the category of manifolds. This allows us to determine the $\EE_\infty$-structures of their truncations $\Pic(\KU)[0,3]$ and $\Pic(\KO)[0,2]$.  It follows that these truncated Picard spaces represent: the Brauer groups of $\ZZ/2$-graded algebra bundles of Donovan-Karoubi, Moutuou and Maycock; the Brauer groups of super 2-lines; and the $K$-theory twists of Freed, Hopkins and Teleman. Our results also imply that that these spaces represent twists of $\String$ and $\Spin$ structures on manifolds and can be used to twist $tmf$-cohomology. Finally, we are able to identify $\pic(\KU)[0,3]$ with a cotruncation of the Anderson dual of the sphere spectrum.
\end{abstract}

\tableofcontents

\section{Introduction}
 In this paper we study the first two $k$-invariants of the Picard spectra of $KO$ and $KU$, or equivalently, the infinite loop space structures on the 2-truncation of $\Pic(\KO)$ and the 3-truncation of $\Pic(\KU)$, which we denote $\GBR$ and $\GBC$ respectively. We will also work with their connected covers which we denote by $\GBRconn$ and $\GBCconn$. These last two spectra are of course equivalent to truncations of $\BGL_1(\KO)$ and $\BGL_1(\KU)$ respectively.  The homotopy types of $\GBR$ and $\GBC$ are not particularly interesting. Indeed, they both split as products of Eilenberg-MacLane spaces: \[\GBC\simeq \ZZ/2\times K(\ZZ/2,1)\times K(\ZZ,3)\] \[\GBR\simeq \ZZ/8\times K(\ZZ/2,1)\times K(\ZZ/2,2).\]
 What we show below however is that neither the first nor second $k$-invariants of the associated Picard \textit{spectra} $\gbc$ and $\gbr$ are trivial. In other words, none of the above splittings are splittings of infinite loop spaces. The first theorems we prove are the computations of these $k$-invariants:
 
\begin{thm*}
The first $k$-invariant of $\gbc$ is $\Sq^2\colon H\ZZ/2\to \Sigma^2H\ZZ/2$. The second $k$-invariant can be taken to be either of the generators of $H^4(\pic_0^1(\KU);\ZZ)\cong\ZZ/4$, both of which restrict to $\beta\circ\Sq^2\colon \Sigma H\ZZ/2\to\Sigma^4H\ZZ$ upon taking a connected cover.
\end{thm*}

\begin{thm*}
The first $k$-invariant of $\gbr$ is $\Sq^2\circ\rho\colon H\ZZ/8\to \Sigma^2H\ZZ/2$,  where $\rho\colon H\ZZ/8\to H\ZZ/2$ is the reduction mod 2 map. The second $k$-invariant is one of two classes in $H^3(\pic_0^1(\KO);\ZZ/2)\cong\ZZ/2\times\ZZ/2$, both of which restrict to $\Sq^2$ upon taking connected covers.
\end{thm*}
\noindent The authors have not been able to find these results in the literature although they do seem to be at least partially known to experts (see e.g.~the proof of \cite[Proposition 7.14]{gepnerlawson}). From these theorems we can deduce the group structures of $\gbr^0(X)$ and $\gbc^0(X)$, which are non-trivial extensions of ordinary singular cohomology groups:
\begin{cor*}
There are bijections of sets
\begin{align*}
\gbr^0(X)&\cong H^0(X;\ZZ/8)\times H^1(X;\ZZ/2)\times H^2(X;\ZZ/2)\\  
\gbc^0(X)&\cong H^0(X;\ZZ/2)\times H^1(X;\ZZ/2)\times H^3(X;\ZZ).
\end{align*}
The group laws on these sets, in the same order, are
\begin{align*}
    (a,b,c)+(a',b',c')&\mapsto (a+a',b+b',c+c'+b\cup b')\\
    (a,b,c)+(a',b',c')&\mapsto (a+a',b+b',c+c'+\beta(b\cup b'))
\end{align*}
where $\beta$ denotes the Bockstein homomorphism. 
\end{cor*}
 
Knowing these group structures allows us to identify other roles that these spectra play in algebraic topology, mathematical physics, and the theory of $C^\ast$-algebras. Almost all of these manifestations are related to twisted $K$-theory but in the literature they are not often directly related to Picard spectra and spaces, which are the universal receptacles for twists of any highly structured cohomology theory. Indeed, because of the infinite loop space splittings $\BGL_1(\KO)\simeq \BGL_1(\KO)[0,2]\times\BGL_1(\KO)[3,\infty]$ and $\BGL_1(\KU)\simeq\BGL_1(\KU)[0,3]\times\BGL_1(\KU)[4,\infty]$, any map from a connected space into $\GBR$ or $\GBC$ gives a purely homotopy theoretic twist of $K$-theory by either composing with the inclusion or multiplying by $-1$ and then including. We now detail the various interpretations of $\gbr$ and $\gbc$.
 
The group laws written above imply that the $\gbr$ and $\gbc$ cohomology groups of a space are isomorphic to well known Brauer groups of $C^\ast$-algebras over that space:
 \begin{thm*}
If $\Br^\infty_{\OO}(X)$ and $\Br^\infty_{\UU}(X)$ are the Brauer groups of (possibly infinite dimensional) real and complex continuous trace graded $C^\ast$-algebras with spectrum $X$ (see \cite{moutuou-gradedBrauergroups, maycockparker}) then there are isomorphisms:
\begin{align*}
    \Br^\infty_{\OO}(X)&\cong\gbr^0(X)\\
    \Br^\infty_{\UU}(X)&\cong\gbc^0(X)
\end{align*}
 \end{thm*}

 \noindent It is a corollary of this fact that $\gbr^0(X)$ and $\mathrm{Tors}(\gbc^0(X))$ are also isomorphic to the graded Brauer groups of Donovan and Karoubi \cite{donovankaroubi}, where $\mathrm{Tors}$ denotes taking the torsion subgroup. Elements of both the $C^\ast$-algebra Brauer groups and the Brauer groups of Donovan and Karoubi are known to produce twists of $K$-theory but our isomorphisms allow one to produce such twists as actual maps of spaces $X\to \BGL_1(\KO)$ and $X\to\BGL_1(\KU)$. In the case that we replace $\gbc$ with $\Sigma^3H\ZZ$, and use \textit{ungraded} $C^\ast$-algebras, the agreement of the two notions of twisted $K$-theory is shown in \cite{hebestreitSagaveTwisted}. Presumably the same is true in our more general case, but we do not prove it in this paper.

Our computations also imply that $\BGL_1(\KO)[0,2]$ and $\BGL_1(\KU)[0,3]$ are equivalent, as infinite loop spaces, to certain fibers in the Postnikov tower for $\BO$.

\begin{thm*}
Let $\BString\to\BSO$ and $\BSpin\to\BO$ be the connective covers in the Postnikov tower of $\BO$. Then there are equivalences of infinite loop spaces:
\begin{align*}
    \GBRconn&\simeq \fib(\BSpin\to\BO)\\
    \GBCconn&\simeq\fib(\BString\to\BSO)
\end{align*}
\end{thm*}

\noindent From this it follows that if $M$ is a manifold (resp.~oriented manifold) then the set of $\Spin$ structures on $M$ (resp.~$\String$ structures) is a torsor for $[M,\GBRconn]\cong H^1(M;\ZZ/2)\ltimes H^2(M;\ZZ/2)$ (resp.~$[M,\GBCconn]\cong H^1(M;\ZZ/2)\ltimes H^3(M;\ZZ)$). This should be compared with the fact that the set of $\Spin$ structures on an oriented manifold is a torsor for $H^2(M;\ZZ/2)$ and the set of $\String$ structures on a $\Spin$ manifold is a torsor for $H^3(M;\ZZ)$ (see \cite[2.11, 2.16]{redden-stringstructures}). It is claimed in \cite[\S 2.1]{FreedDistlerMoore-spinstructures} that $\Spin$ structures are a torsor for $H^1(M;\ZZ/2)\times H^2(M;\ZZ/2)$ but a proof is not included. Moreover, because there are equivalences of infinite loop spaces $\GBCconn\simeq \BGL_1(\KO[0,1])$ and $\GBCconn\simeq\BGL_1(\KU[0,2])$ we can twist $\Spin$ structures (resp.~$\String$ structures) by bundles of $\KO[0,1]$ (resp.~$\KU[0,2]$) modules, which we can interpret as real and complex super 2-line bundles respectively.

Not surprisingly, our work here is also related to the work on twisted $K$-theory and mathematical physics by Freed, Hopkins, Teleman and others \cite{FreedDistlerMoore-spinstructures, freedOrientifolds,freed-anomalies, FHT-consistent, FHT1}.

\begin{thm*}
Let $cAlg_{\RR}^\times$ and $cAlg_{\CC}^\times$  be the spectra associated to the Picard 2-groupoids of invertible topological $\RR$ and $\CC$-superalgebras respectively. Then there are equivalences of spectra:
\begin{align*}
    cAlg_{\RR}^\times&\simeq\gbr\\
    cAlg_{\CC}^\times&\simeq\gbc
\end{align*}
\end{thm*}
\noindent It follows that $\gl_1(\KO)[0,1]$ and $\gl_1(\KU)[0,2]$ are the Picard spectra of real and complex superlines respectively. Again using the fact that $\gl_1(\KO)[0,2]\simeq \gl_1(\KO[0,2])$ and $\gl_1(\KU)[0,3]\simeq\gl_1(\KU[0,3])$, we have that a bundle of real, resp.~complex, superlines on a space $X$ is the same data as a bundle of $\KO[0,2]$-lines, resp.~$\KU[0,3]$-lines. This also identifies $\gbc$ with the spectrum $R_{-1}$ of \cite{FreedDistlerMoore-spinstructures} whose associated cohomology theory is proposed as the container for the flux of the oriented superstring B-field.

The above theorem can be restated in the context of the $K$-theory twists of \cite[1.80]{freedOrientifolds} and \cite[Corollary 2.25]{FHT1} assuming we take $X$ to be only a space rather than a topological groupoid.

\begin{thm*}
Let $\pi_0\mathfrak{Twist}_{\KU}(X)$ and $\pi_0\mathfrak{Twist}_{\KO}(X)$ be the groups of isomorphism classes of $\KU$ and $\KO$ twists on $X$ in the sense of ibid. Then there are isomorphisms of groups:
\begin{align*}
   \pi_0\mathfrak{Twist}_{\KO}(X)&\cong \gbr^0(X)\\
   \pi_0\mathfrak{Twist}_{\KU}(X)&\cong\gbc^0(X)
\end{align*}
\end{thm*}

In related work by the same authors, the Anderson dual of the sphere spectrum, $I_\ZZ$, often arises (see for instance \cite[Hypothesis 5.17]{Freed-SRE} and \cite[Theorem 5.27]{FreedHopkins-reflectionpositivity}). We show that, at least in the complex case, the truncated Picard spectrum is closely related:
\begin{thm*}
There is an equivalence of spectra \[\Sigma^3(I_\ZZ[-3,\infty))\simeq\gbc.\]
\end{thm*}
\noindent This suggests a connection between invertible topological field theories (and deformations thereof) and bundles of truncated $\KU[0,2]$-lines.

\subsection{Conventions and Notation}

In this paper we will often work with $\infty$-categories (in the sense of \cite{htt}) of \textit{spectra} and \textit{$\infty$-groupoids}, which we denote by $\Sp$ and $\Spaces$ respectively. There has been some disagreement lately about an efficacious term for the objects of the $\infty$-category $\Spaces$. We find the term ``$\infty$-groupoid'' to be too long, the term ``anima'' to be unpleasant to pluralize, and the term ``space'' to be far too ambiguous. Therefore, going forward, we will use ``$h$-type'' to refer to these structures and propose this as an alternative to the terms listed above. We will still occasionally call them $\infty$-groupoids when we want to emphasize their use as $\infty$-categories in which all morphisms are invertible. When we use the term ``space'' we will specifically be referring to a compactly generated, weakly Hausdorff topological space, the category of which we will denote by $\Top$. Recall that if $E\in\Sp$ is a spectrum then it has an associated cohomology theory $E^\ast$. This cohomology theory is defined equally well on $\Spaces$ and $\Top$, so we will apply it to both sorts of objects without comment.

The term ``infinite loop space'' frequently denotes what, in the language of \cite{ha}, one might call ``grouplike $\EE_\infty$-monoids in $\Spaces$.'' Another commonly used term for such structures, and the one we will employ, is ``$\EE_\infty$-group.'' We emphasize however that being an $\EE_\infty$-group is a \textit{structure} rather than a property. If we wish to refer to an $h$-type which has an $\EE_\infty$-structure without concerning ourselves with whether or not it is grouplike, we will often say ``$\EE_\infty$-type'' instead of the unwieldy ``$\EE_\infty$-$h$-type.''

Many of our constructions will involve a field $\FF$. We will always assume that $\FF$ has a topology which is accounted for by these constructions, e.g.~vector bundles. We will occasionally wish to consider a field with its discrete topology, in which case we will say so.

We will often be interested in truncations and connective covers of spectra and $h$-types. For integers $n,m\geq 0$ and a spectrum or $h$-type $X$ we will write $X[n,m]$ for the $n$-connective cover of the $m$-truncation of $X$. Note the use of connective here, as opposed to \textit{connected}, which differs by 1. In the special, and ubiquitous, case that $X$ is $\Pic(R)$ or $\pic(R)$ for a commutative ring spectrum $R$, we will use the non-standard notation of Definition \ref{def:truncatedPicardthings}. We do this to keep the names of these frequently used objects compact. 

\subsection{Acknowledgements} 

The authors thank Tyler Lawson, Leon Liu, Natalia Pacheco-Tallaj, Eric Peterson (especially Eric Peterson) and Charles Rezk for helpful conversations regarding this material. 

\section{Background}

In this section we recall, for an $\EE_\infty$-ring spectrum $R$, the construction of the so-called \textit{Picard space} of $R$ along with various spectra and $h$-types which can be built from it. More detailed constructions of these can be found in \cite{abghr,abg}. 

\begin{defn}\label{defn:Picardthings}
Let  $R$ be an $\EE_\infty$-ring spectrum. Then we make the following definitions.
\begin{enumerate}
\item We write $\Pic(R)$ for the maximal $\infty$-groupoid in $\LMod_R(\Sp)$ spanned by modules which are invertible with respect to the tensor product over $R$. Recall that $\Pic(R)$ is an $\EE_\infty$-group with base point $R\in\Pic(R)$. We denote its infinite delooping by $\pic(R)$, i.e.~$\Omega^\infty\pic(R)\simeq\Pic(R)$.
\item We let $\GL_1(R)$ denote the pullback of the projection $\Omega^\infty R\to \pi_0(R)$ along the inclusion $\pi_0(R)^\times\hookrightarrow \pi_0(R)$. Recall that $\GL_1(R)$ is equivalent, as an $\EE_\infty$-group, to $\Omega\Pic(R)$. Equivalently, we could take $\BGL_1(R)$ to be the base point component of $\Pic(R)$. We denote the infinite deloopings of $\GL_1(R)$ and $\BGL_1(R)$ by $\gl_1(R)$ and $\bgl_1(R)$ respectively.
\end{enumerate}
\end{defn}

\begin{defn}\label{def:truncatedPicardthings}
Given an $\EE_\infty$-ring $R$, we will write $\Pic_n^m(R)$, respectively $\pic_n^m(R)$, for the $m$-truncated, $n$-connective covers of $\Pic(R)$ and $\pic(R)$ . In other words, $\Pic_n^m(R)$ and $\pic_n^m(R)$ are equivalent to $\Pic(R)$ and $\pic(R)$ respectively in homotopy degrees $n$ through $m$, and have trivial homotopy groups elsewhere. We will use $\bgl_1(R)[0,n]$ and $\BGL_1(R)[0,n]$ interchangeably with $\pic_1^n(R)$ and $\Pic_1^n(R)$.
\end{defn}

\begin{rmk}
Note that both $m$-truncating and taking $n$-connective covers determine symmetric monoidal functors $\Spaces\to\Spaces$ for all $n$ and $m$, so $\Pic_n^m(R)$ has an $\EE_\infty$-group structure induced by that of $\Pic(R$).
\end{rmk}

It will also be useful to have the following definition recorded here, though we will not make frequent use of it. This definition appears, e.g., in \cite{hopkinslurie-brauer}.

\begin{defn}\label{defn:BrauerSpace}
    Let $R$ be a commutative ring spectrum. Then $\LMod(R)$ is an $\EE_\infty$-algebra in the symmetric monoidal $\infty$-category of $\infty$-categories, $\Cat_\infty$, which in turn has its own category of modules. We write $\Br(R)$ for the $\EE_\infty$-group of invertible $\LMod(R)$-modules and equivalences between them. 
\end{defn}

\section{Non-triviality of $k$-invariants}

In this section we show that the $k$-invariants of $\pic_0^1(\KO)$, $\pic_0^1(\KU)$, $\gbrconn$ and $\gbcconn$ are non-trivial. This will follow from showing that these spectra can be modeled by Picard groupoids, or sheaves thereof, that have non-trivial symmetries.

\subsection{Picard groupoids and symmetries}
Recall that there is an equivalence between spectral 2-types, i.e.~spectra with homotopy groups only in degrees $n$ and $n+1$ for $n\in\ZZ$, and symmetric monoidal categories in which every morphism is invertible and every object has a tensor inverse, i.e.~Picard groupoids (see \cite{johnsonosornoPicard} or \cite[Proposition B.12]{HopkinsSinger-quadratic}). In particular, Picard groupoids are equivalent to spectra with non-trivial homotopy groups only in degrees 0 and 1. Given a Picard groupoid $\mathcal{P}$, the associated spectrum has $\pi_0$ equal to the group of isomorphism classes of objects of $\mathcal{P}$. By using the Picard groupoid structure one can show that every object of $\mathcal{P}$ has isomorphic group of automorphisms.  The first homotopy group of its associated spectrum can then be taken to be the group of automorphisms of any object of $\mathcal{P}$. 

The single $k$-invariant of the spectrum associated to $\mathcal{P}$ is determined by the symmetry natural isomorphism of the symmetric monoidal structure on $\mathcal{P}$. One can use the Picard groupoid structure of $\mathcal{P}$ to show that this natural isomorphism is in turn entirely determined by the symmetry isomorphism of the tensor square of the tensor unit. The possible such isomorphisms, hence isomorphism classes of symmetric monoidal structures, are given by $\Hom_{\mathrm{Ab}}(\pi_0\otimes\ZZ/2,\pi_1)\cong H^2(\pi_0\otimes\ZZ/2;\pi_1)$. The $k$-invariant of the associated spectrum is trivial exactly when this symmetry isomorphism is the identity, i.e.~the zero map $\ZZ/2\to \pi_1$.

In what follows, we show that the first $k$-invariants of $\pic_0^1(\KO)$ and $\pic_0^1(\KU)$ are non-trivial. These are both stable 2-types so it suffices to show that the symmetry isomorphisms of the associated Picard groupoids are non-trivial. In fact, we only need to consider the symmetry isomorphism of the sphere spectrum. Note that this data is determined at the level of homotopy categories.

\begin{lem}\label{lem:nontriviality of first pics k invariant}
The $k$-invariant of the Picard spectrum $\pic_0^1(\sph)$ is non-trivial.
\end{lem}

\begin{proof}
Let $\sigma\colon\sph\otimes\sph\to\sph\otimes\sph$ be the symmetry map of $\sph$ in $\Sp$. This map is the stabilization of the ``swap'' equivalence $S^1\wedge S^1\to S^1\wedge S^1$ in $\pi_2(S^2)^\times$, which is not homotopic to the identity. 
\end{proof}

\begin{lem}\label{lem:nontrivialityoffirstk}
The functors $\pic(\sph)\to\pic(\KO)$ and $\pic(\sph)\to\pic(KU)$ given by tensoring with $\KO$ and $\KU$ respectively, equivalently the maps induced by applying the Picard spectrum functor to the units $\sph\to \KO$ and $\sph\to\KU$, are surjective on $\pi_0$ and isomorphisms on $\pi_1$. 
\end{lem}

\begin{proof}
We only prove the statement for $\KO$ as the proof for $\KU$ is identical. Recall from \cite{gepnerlawson} that the objects of $\pic(\KO)$ are precisely the shifts $\Sigma^i \KO$ for $0\leq i\leq 7$, using the eightfold periodicity of $\KO$. Each of these is covered by $\Sigma^i\sph$ for $0\leq i\leq 7$ under the given functor $\pic(\sph)\to\pic(\KO)$, so the functor induces a surjection on $\pi_0$. 

On $\pi_1$ this is precisely $\pi_0$ of the induced map $\gl_1(\sph)\to\gl_1(\KO)$. Recall that the unit map $\sph\to \KO$ is an isomorphism on $\pi_0$, for instance because it is a ring map $\ZZ\to\ZZ$. Now using the fact that $\gl_1(R)$ is the infinite delooping of the pullback of $\pi_0(R)\to\pi_0 R\xleftarrow{}\Omega^\infty R$ for any commutative ring spectrum $R$, we have that the induced map $\gl_1(\sph)\to\gl_1(\KO)$ is also an isomorphism $\pi_0$. 
\end{proof}

\begin{cor}
    The first $k$-invariants of $\pic(\KO)$ and $\pic(\KU)$ are non-trivial.
\end{cor}

\subsection{Chain Bundle Models of Topological $K$-theory}\label{sec:SegalGroupoid}

In \cite[Section 6]{BNV-differentialcoh} it is shown that the commutative ring spectrum $ku$ can be recovered as the homotopification of the sheaf of groupoids of complex vector bundles, suitably group completed. The constructions below are similar, but replace vector bundles with \textit{$\ZZ/2$-graded chain complexes} of vector bundles. This sheaf is still not homotopy invariant, so we must localize it with respect to $\mathbb{R}$ (i.e.~make it \textit{concordance invariant}). The resulting localized sheaf represents complex topological $K$-theory by \cite[Appendix I]{segalequivariant}. Moreover, since it is concordance invariant, the results of \cite{ADH-differential} imply that it is a constant sheaf determined by its value at the point, which we show to be $\ku$. This sheaf then has a subsheaf of \textit{invertible} $\ZZ/2$-graded chain complexes (with respect to the usual tensor product) which evaluates to $\gl_1(ku)\simeq\gl_1(\KU)$ on the point. 

These sheaves are slightly more complicated than those of \cite{BNV-differentialcoh} but give us better access to the $k$-invariants of $\gl_1(\KU)$ and therefore the $k$-invariants of $\pic(KU)$. Of course all of our arguments apply equally well to case of real $K$-theory. It is worth noting that the sheaf constructed in \cite{BNV-differentialcoh} returns $K(\CC)$ when evaluated at the point before localizing, whereas ours most certainly does not. We suspect that $K(\CC)$ however, or $K(\FF)$ in general for a discrete field $\FF$, can be recovered by a  variant of our construction using principal $\GL(\FF)$-bundles instead of vector bundles, i.e.~by always equipping $\FF$ with the discrete topology (cf.~Conjecture \ref{rmk:k invariants of K of discrete fields}).

Throughout this section, $Ch^{perf}_{\FF}\colon \Top^{op}\to \Gpd$ will denote the sheaf whose value at a space is the groupoid of bounded chain complexes of finite dimensional $\FF$-vector bundles on $X$ with homotopy classes of homotopy equivalences between them. We assume all the standard structures of this sheaf, e.g.~tensor products and direct sums. For our purposes it will be convenient to work with a slightly different sheaf, which we define below. 

\begin{defn}\label{def:DSVs} Let $\FF$ be a field. We begin by defining chain complexes of $\FF$ vector spaces which are graded by $\ZZ/2$. We call these \textit{differential super $\FF$-vector spaces} or DSVs for short.
\begin{enumerate}
    \item A differential super $\FF$-vector space is a $\ZZ/2$-graded $\FF$-vector space $V=V_0\oplus V_1$ equipped with two maps $d_0\colon V_0\to V_1$ and $d_1\colon V_1\to V_0$ such that $d_0d_1=d_1d_0=0$. A morphism of differential super $\FF$-vector spaces is a morphism of $\ZZ/2$-graded vector spaces which commutes with the differential. These data form a category which we denote by $\CG$.  

    \item If $V=(V_0\oplus V_1, d_0, d_1)$ and $W=(W_0\oplus W_1, d'_0,d'_1)$ are DSVs, we will write $V\otimes W$ for the DSV which is $(V\otimes W)_0=(V_0\otimes W_0)\oplus (V_1\otimes W_1)$ in degree zero and $(V\otimes W)_1=(V_1\otimes W_0)\oplus (V_0\otimes W_1)$ in degree one. The differentials are given in the usual way after reducing all indices modulo 2.
    
    \item For $V=(V_0\oplus V_1,d_0,d_1)$ a DSV, we define $H_0(V)=\ker(d_0)/\im(d_1)$ and $H_1(V)=\ker(d_1)/\im(d_0)$. We say that a morphism of DSVs is a quasi-isomorphism if it induces isomorphisms of these two homology groups.

    \item Let $V=(V_0\oplus V_1,d_0,d_1)$ and $W=(W_0\oplus W_1,d_0',d_1')$ be DSVs. Given two maps $f,g\colon V\to W$ of DSVs we say that a chain homotopy between $f$ and $g$ is a pair of maps $h_0\colon V_0\to W_1$ and $h_1\colon V_1\to W_0$ such that $f_0-g_0=d_1'h_0+h_1d_0$ and $f_1-g_1=d_0'h_1+h_0d_1$. If there is a chain homotopy between $f$ and $g$ we write $f\sim g$.
    
    \item We say that a map $f\colon V\to W$ of DSVs is a homotopy equivalence if there exists $g\colon W\to V$ and chain homotopies $f\circ g\sim id_{W}$ and $g\circ f\sim id_{V}$. 
    
    \item We write $\varepsilon\colon Ch^{perf}_\FF\to DSV_{\FF}$ for the functor which takes a bounded and finite dimensional $\FF$-chain complex $(E_\bullet,\partial)$ to the  DSV whose graded vector space is $E=\left(\oplus_{i}E_{2i}\right)\bigoplus\left(\oplus_{i}E_{2i+1}\right)$ and whose differential is the obvious restriction of $\partial$. 

    \end{enumerate}
\end{defn}

    We leave it to the reader to check the following lemmas which are standard arguments in homological algebra.

    \begin{lem}\label{lem:DSVquasiisos}
        A map $f\colon V\to W$ is a quasi-isomorphism of DSVs if and only if it is a homotopy equivalence.
    \end{lem}

    \begin{lem}
        The tensor product of DSVs defines a symmetric monoidal structure on $\CG$ with monoidal unit $(\FF\oplus 0, 0, 0)$. This tensor product distributes over the direct sum of DSVs.
    \end{lem}
    
    \begin{lem} The functor $\varepsilon$ is full, surjective on objects, and symmetric monoidal. Moreover, it takes quasi-isomorphisms to quasi-isomorphisms and chain homotopies to chain homotopies.
    \end{lem}

    \begin{rmk}
        As with the classical case, one direction of Lemma \ref{lem:DSVquasiisos} depends on $\FF$ being a field. For a general ring, homotopy equivalences are strictly stronger than quasi-isomorphisms.
    \end{rmk}

\begin{defn}
Let $X$ be a space and $\mathbb{F}$ a field. We write $\CG(X)$ for the groupoid of bundles of DSVs on $X$ whose morphisms are chain homotopy classes of homotopy equivalences. Because bundles can be pulled back along maps of spaces this defines a functor $\CG\colon \Top^{op}\to \Gpd$. The tensor product and direct sum bundles are defined in the usual way. 
\end{defn}

   Going forward, many of our results will apply equally well to the functors $\CG$ and $Ch^{perf}_{\FF}$. We will therefore write $\scrF$ to denote either.  

The following lemma says that $\scrF$ is a sheaf of \textit{ring groupoids} in the sense of \cite{drinfeld-ringGroupoids}, which are a special case of the ring categories of \cite{laplaza-ringcats}. These are essentially commutative monoids in the category of Picard groupoids. Without loss of generality ring groupoids can always be assumed to have underlying \textit{strict} Picard groupoid.

\begin{lem}\label{lem: ring structure on CG}
The functor $\scrF$ is a presheaf of grouplike symmetric monoidal groupoids with respect to direct sum of chain bundles, a presheaf of symmetric monoidal groupoids with respect to tensor product of chain bundles, and the latter structure distributes over the former. Moreover, the natural transformation $\varepsilon(X)\colon Ch_{\FF}^{perf}(X)\to \CG(X)$ preserves this structure.
\end{lem}

\begin{proof}
The lemma follows from standard arguments for bundles extended suitably to chain complexes. We only remark that the \textit{grouplike} structure of these groupoids, with respect to direct sum, comes from the fact that we have taken DSVs and chain complexes rather than vector bundles and for an object $E_\bullet$ of either type we can always find $E'_\bullet$ such that $E_\bullet\oplus E'_\bullet$ is acyclic.
\end{proof}

\begin{rmk}
In the usual construction of $K$-theory, it is necessary to take a group completion of the monoid, under direct sum, of vector bundles on $X$. Working with bundles of DSVs, or chain complexes, makes this unnecessary because there are already ``positive'' and ``negative'' classes.
\end{rmk}

\begin{rmk}\label{rmk:MFDrestrict}
It will be convenient to extend the codomain and restrict the domain of $\scrF$. By taking nerves there is an inclusion $\Gpd\subset \Spaces$ under which the grouplike symmetric monoidal structure of Lemma \ref{lem: ring structure on CG} makes $\scrF$ into a presheaf of connective spectra. The tensor product symmetric monoidal structure of Lemma \ref{lem: ring structure on CG} makes it into a sheaf of connective commutative ring spectra. For convenience, we will also restrict $\scrF$ to the subcategory of smooth manifolds and smooth maps $ \Mfd^{op}\subset\Top^{op}$.
\end{rmk}

\begin{lem}\label{lem: CG is a sheaf}
$\scrF$ is a sheaf of connective commutative ring spectra on $\Mfd$ with respect to coverings by families of open embeddings. 
\end{lem}

\begin{proof}
To prove that $\scrF$ is a sheaf, it suffices to show that, for a fixed manifold $M$, it satisfies descent on the ``little'' site of open submanifolds of $M$ (indeed by \cite[Lemma 3.5.3]{ADH-differential} it suffices to check only on Euclidean spaces). This follows immediately from the definitions, as bundles themselves are defined locally. Further, $\scrF$ is valued in \textit{grouplike symmetric monoidal} $\infty$-groupoids (i.e.~connective spectra) as a result of Lemma \ref{lem: ring structure on CG} and the fact that the inclusion $\Gpd\subset\Spaces$ is symmetric monoidal (with respect to the Cartesian product). It remains a sheaf even when valued in commutative ring spectra since limits of commutative ring spectra are computed on underlying spectra (cf.~\cite[3.2.2.5]{ha}).
\end{proof}

\noindent We recall the $\mb R$-invariantization functor from Definition 4.2.5 and Proposition 5.1.2 of \cite{ADH-differential}.

\begin{defn}
    Let $\Delta_{sm}^n$ denote the hyperplane in $\mathbb{R}^{n+1}$ spanned by points $(x_1,x_2,\ldots,x_{n+1})$ such that $\sum_{i=1}^{n+1}x_i=1$, also known as the \textit{smooth $n$-simplex}. Let $\Delta_{sm}^\bullet$ denote the cosimplicial manifold which is $\Delta^n_{sm}$ is degree $n$. Its coface maps $\Delta^n_{sm}\to\Delta^{n+1}_{sm}$ are the $n+2$ inclusions $\Delta^n_{sm}\to \Delta^{n+1}_{sm}$ given by the $n+2$ possible intersections of $\Delta^{n+1}_{sm}$ with the coordinate hyperplanes. The codegeneracies $\Delta^{n+1}_{sm}\to\Delta^{n}_{sm}$ are given by the $n$ possible ways of adding adjacent coordinates.
\end{defn}

\begin{defn}\label{defn:Rinvariant}
Let $\mathscr{E}\colon \Mfd^{op}\to\Top$ be a presheaf of $h$-types. Then define $\LR\mathscr{E}\colon \Mfd^{op}\to \Top$ to be the presheaf of $h$-types whose value at $X$ is given by $\LR\mathscr{E}(X)=|\mathscr{E}(X\times\Delta_{alg}^\bullet)|$.
\end{defn}

\begin{rmk}\label{rmk: what does L do}
If $a$ and $a'$ are objects of $\scrF(X)$ (thought of as a Picard groupoid), $b$ is an object of $\scrF(X\times\Delta^1)$, and we have isomorphisms $d_0(b)\cong a$ and $d_1(b)\cong a'$, where $d_0$ and $d_1$ are the face maps, then $a$ and $a'$ are equivalent in $\LR\scrF(X)$. Similar statements hold for the $\scrF(X\times\Delta^n)$ and the relevant higher face maps.  This has the effect of making $\scrF$ insensitive to the difference between a space $X$ and its ``stabilizations'' $X\times\mb R\simeq X\times\Delta^1$, $X\times\mb R^2\simeq X\times \Delta^2$, and so forth.
\end{rmk}

\begin{rmk}\label{rmk:LRsameasSegal}
In the Appendix of \cite{segalequivariant}, working with $Ch^{perf}_{\CC}$, Segal uses an equivalence relation to identify bundles which can be homotoped into one another along $X\times\mathbb{R}\cong X\times\Delta^1$. The above construction can be thought of actually inserting the homotopies (and homotopies between homotopies etc.)~between such bundles, forcing the sheaf to be insensitive to deformations of $X$. However, after taking connected components the resulting group is the same as Segal's (see Theorem \ref{prop:CGisKtheory} below). 
\end{rmk}

\begin{lem}
The presheaf $\LR\scrF$ is a sheaf. 
\end{lem}

\begin{proof}
See Remark 4.2.6 \cite{ADH-differential} or Proposition 2.6 of \cite{BNV-differentialcoh}.
\end{proof}

\begin{prop}\label{prop:CGisKtheory}
There is a group isomorphism $\pi_0\LR \CGC(X)\cong KU^0(X)$ which is natural in $X$.
\end{prop}

\begin{proof}

First note that by \cite[3.2.3.1]{ha} and \cite[1.4.3.9]{ha} the simplicial colimits defining $\LR\CGC(X)$ and $\LR Ch^{perf}_{\FF}$ can be taken in spaces instead of in connective $\EE_\infty$-rings. Now, $\pi_0$ of a simplicial colimit depends only on the subdiagram (which is a coequalizer) involving the 0- and 1-simplices. From the definition of $\LR Ch^{perf}_{\CC}(X)$, $\pi_0$ of that coequalizer is the set of chain-bundles on $X$ up to concordance. Remark \ref{rmk:LRsameasSegal} therefore implies that \cite[Proposition A.I]{segalequivariant} (in which ``concordant" is called ``homotopic") gives an isomorphism $\pi_0\LR Ch_{\CC}^{perf}(X)\xrightarrow{\cong} \KU^0(X)$. 

This isomorphism is given by taking a bundle to its Euler characteristic. Specifically, a bundle of chain complexes $E_\bullet$ is taken to the alternating sum of the $K$-theory classes of each grade, $\Sigma_{i\in\ZZ}(-1)^i[E^i]$. This clearly factors through $\pi_0(\varepsilon(X))$ so that we have a composite isomorphism \[\pi_0\LR Ch_{\CC}^{perf}(X)\xrightarrow{\varepsilon}\pi_0\LR\CGC(X)\to KU^0(X)\] in which the last morphism forgets the differential. This implies that $\varepsilon$ is injective. We have already seen that it is surjective, so it is an isomorphism.  
\end{proof}

\begin{thm}\label{thm:LRDSV is ku}
There is an equivalence of commutative ring spectra $\LR\CGC(\ast)\simeq \ku $.
\end{thm}

\begin{proof}
This follows from \cite[Proposition 4.3.1, Proposition I.2]{ADH-differential} and Proposition \ref{prop:CGisKtheory}.  Specifically, we know that evaluation at a point is an equivalence between concordance invariant sheaves of spectra on $\Mfd$ and spectra. The inverse of this equivalence is given by taking the constant sheaf. Moreover, given a spectrum $E$, the value of $\mathrm{const}(E)$ at a manifold $X$ is $\Map(\Sigma^\infty_+X,E)$.   As a result, $\LR\CGC(X)\simeq \Map(\Sigma^\infty_+X, \LR\CGC(\ast))$ and thus $\pi_0(X,\LR\CGC(\ast))\cong KU^0(X)$ for every manifold $X$, so $\LR\CGC(\ast)\simeq ku$ by Brown representability.
\end{proof}

The above arguments, along with those of \cite{segalequivariant}, can be repeated mutatis mutandis with chain complexes of $\mathbb{R}$-vector spaces rather than $\mb C$-vector spaces. This leads to:

\begin{thm}
There is an equivalence of commutative ring spectra $\LR\CGR(\ast)\simeq  \ko$.
\end{thm}

\begin{cor}\label{cor:gl1DSVC gives gl1KU}
Consider the full symmetric monoidal subgroupoid $\gl_1\!\CGC(X)$ of $\CGC(X)$ on the objects which are invertible in the tensor product monoidal structure.  Then $\gl_1\!\CGC$ is a sheaf of connective spectra  and $\LR\gl_1\!\CGC(\ast)\simeq\gl_1(\ku)\simeq \gl_1(\KU)$.
\end{cor}

\begin{proof}
The subgroupoid $\gl_1\!\CGC(X)$ is the full subgroupoid of $\CGC(X)$ on bundles $E_\bullet$ such that $\mathrm{dim}(E_0)-\mathrm{dim}(E_1)=\pm 1$. The same argument for Lemma \ref{lem: CG is a sheaf} shows that $\gl_1\!\CGC$ is a sheaf of connective spectra (but not ring spectra). Let $\mathbb{Z}(-)$ denote the sheaf whose value at $X$ is the discrete groupoid of continuous $\mathbb{Z}$-valued functions on $X$. Similarly let $\mathbb{Z}/2(X)$ be the discrete groupoid of continuous ${\pm1}$-valued functions on $X$. The inclusion $i_X\colon\gl_1\!\CGC(X)\hookrightarrow\CGC(X)$ fits (essentially by definition) into a pullback square of $h$-types
$$\begin{tikzcd}
\gl_1\!\CGC(X)\arrow[r,"i_X"]\arrow[d]&\CGC(X)\arrow[d] \\
 \mathbb{Z}/2(X)\arrow[r,hook]& \mathbb{Z}(X)
\end{tikzcd}$$
where the vertical maps send a DSV to its graded dimension. This extends to a levelwise pullback square of simplicial $h$-types
$$\begin{tikzcd}
\gl_1\!\CGC(X\times\Delta_{alg}^\bullet)\arrow[r,"i_X^\bullet"]\arrow[d]&\CGC(X\times\Delta_{alg}^\bullet)\arrow[d] \\
 \mathbb{Z}/2(X\times\Delta_{alg}^\bullet)\arrow[r,hook]& \mathbb{Z}(X\times\Delta_{alg}^\bullet).
\end{tikzcd}$$
Forgetting the monoidal structure and considering this as just a diagram of $h$-types, the diagram is still a pullback after applying $\LR$ (see \cite[Definition 1.1, Proposition 5.4]{Rezkcompatible}).  When $X=\ast$, this gives the pullback diagram of $h$-types
$$\begin{tikzcd}
\Omega^\infty\LR\gl_1\!\CGC(\ast)\arrow[r,""]\arrow[d]&\Omega^\infty\LR\CGC(\ast)\simeq\mathbb{Z}\times BU\arrow[d] \\
 \mathbb{Z}/2\arrow[r,hook]& \mathbb{Z}
\end{tikzcd}$$ exhibiting $\Omega^\infty\LR\gl_1\!\CGC(\ast)$ as $GL_1\ku$. Since $i_X$ is symmetric monoidal, there is a natural-in-$Y$ isomorphism of Abelian groups $\LR\gl_1\!\CGC(\ast)^0(Y)\cong Map(\Sigma_+^\infty Y, \LR\gl_1\!\CGC(\ast)) \cong ku^0(Y)^\times$  for any space $Y$ (where the right side has the tensor product Abelian group structure). Hence by Brown representability there is an equivalence of spectra $\LR\gl_1\!\CGC(\ast)\simeq\gl_1(\ku)$.
\end{proof} 

The same arguments apply to the real case, which gives Corollary \ref{cor:gl1DSVR gives gl1KO}. Our arguments also seem likely to apply in the case that $\mathbb{F}$ is a discrete field $\mathbb{F}$ which we codify in Conjecture \ref{rmk:k invariants of K of discrete fields}.

\begin{cor}\label{cor:gl1DSVR gives gl1KO}
Consider the full symmetric monoidal subgroupoid $\gl_1\!\CGR(X)$ of $\CGR(X)$ on the objects which are invertible in the tensor product monoidal structure.  Then $\gl_1\!\CGR$ is a sheaf of connective spectra  and $\LR\gl_1\!\CGR(\ast)\simeq\gl_1(\ko)\simeq \gl_1(\KO)$.
\end{cor}

\subsection{The Groupoid of $\ZZ/2$-graded Line Bundles}\label{sec:gradedlinebundles}
We now define a simpler sheaf that will map to $\CG$ and help us to understand its structure. The sheaves $\LR\mathscr{L}^\CC$ and $\LR\mathscr{L}^\RR$ that we describe here will end up being truncations of $\LR\gl_1\!\CGC\simeq \gl_1(KU)$ and $\LR\gl_1\!\CGR\simeq \gl_1(KO)$ after evaluating at the point. Some of the ideas of this section exist in \cite{freedOrientifolds} but we go a step further in relating these structures to the unit spectra of $K$-theory. This is in constrast to \cite[1.45]{freedOrientifolds} in which they are described as truncations of $\ko$ itself. Our results are arguably more conceptually satisfying given that these structures are in fact used for twisting both real and complex $K$-theory. Indeed, in \cite{freedOrientifolds}, Freed remarks that he does not have a conceptual reason for the appearance of $\ko$ in his constructions.

\begin{defn}\label{defn:FHT}
Let $X$ be a space, and $\mb F$ a (topological) field. Let $\FHT(X)$ be the groupoid of pairs $(\xi,n)$ where $\xi\colon E\to X$ is an $\mb F$-line bundle and $n\colon X\to \ZZ/2=\{0,1\}$ is a continuous function. The morphisms between $(\xi,m)$ and $(\xi',m')$ in $\FHT(X)$ will be the empty set if $n\neq m$ and the set of bundle isomorphisms otherwise. We will refer to $\FHT(X)$ as the groupoid of  $\FF$-superline bundles on $X$.
\end{defn}

\begin{defn}
Given two pairs $(\xi,n),(\xi',m)\in\FHT(X)$, we define a symmetric monoidal structure on $\FHT(X)$ by declaring that:
\begin{enumerate}
    \item The tensor product $(\xi,n)\otimes(\xi',m)$ is the object $(\xi\otimes\xi',n+m)$, where the tensor product of the left hand coordinate is the standard tensor product of principal bundles. 
    \item The symmetry isomorphism $(\xi,n)\otimes(\xi',m)\to (\xi',m)\otimes(\xi,n)$ is given on the fiber over $x$ by $(v,w)\mapsto (w, (-1)^{n(x)m(x)}v)$. 
\end{enumerate}
\end{defn}

\begin{rmk}
    The above definition is almost identical to the sheaf of graded $\mathbb{T}$-bundles $\mathcal{B}\mathbb{T}^{\pm}$ in \cite[Definition 2.1]{FHT1}. Their definition differs from ours only in their definition of the symmetry isomorphism which would be written in our notation as $(v,w)\mapsto (w, vn(x)m(x))$. This makes sense if we take $\ZZ/2=\{-1,1\}$, but then the second coordinate in the tensor product formula, i.e.~$n+m$, does not make sense.
\end{rmk}

\begin{lem}\label{lem: FHTX does not decompose}
The groupoid $\FHT(X)$ does not decompose as a product of symmetric monoidal groupoids.
\end{lem}
\begin{proof}
Let $\mathbb{Z}/2(X)$ be the discrete symmetric monoidal groupoid of $\mathbb{Z}/2$-valued functions with the pointwise group structure. The sign in the second component of the symmetry isomorphism of $\FHT(X)$ prevents the natural projection $\FHT(X)\rightarrow\mathbb{Z}/2(X)$ from having a symmetric monoidal section.
\end{proof}

\begin{rmk}
The argument in the proof of Lemma \ref{lem: FHTX does not decompose}, along with \cite{johnsonosornoPicard}, also shows that the connective spectrum associated to $\FHT(X)$ has non-trivial $k$-invariant, though we will not need this fact.
\end{rmk}

\begin{rmk}
Definition \ref{defn:FHT} extends to a presheaf of grouplike symmetric monoidal groupoids on $\text{Top}$ whose domain, following Remark \ref{rmk:MFDrestrict}, we restrict to $\Mfd^{op}$ and whose codomain we extend to $\Sp_{\geq0}$. 
\end{rmk}

\begin{lem}
The presheaf $\FHT$ is a sheaf of connective spectra on $\Mfd$. 
\end{lem}

\begin{proof}
Because the inclusion $\Gpd\subset\Spaces$ is symmetric monoidal (and the category of connective spectra is equivalent to the category of grouplike $\EE_\infty$ $h$-types), it suffices to show that $\FHT$ is a sheaf of grouplike symmetric monoidal groupoids. The fact that it is a sheaf of groupoids, without symmetric monoidal structure, follows immediately from the fact that it decomposes as a sum of presheaves of groupoids which are clearly sheaves. The symmetric monoidal structure glues as well since limits of symmetric monoidal groupoids are computed in $\Cat$. The grouplike condition is certainly satisfied for each $X\in \Mfd$.
\end{proof}

\begin{prop}\label{prop:FHTinCG}
With respect to the symmetric monoidal structures given by bundle tensor product, there is a natural in $X$ symmetric monoidal functor $\ell(X)\colon\FHT(X)\to\CG(X)$ defined by taking $(\xi,n)$ to the chain bundle which has line bundle $\xi$ concentrated in degree $n$.
\end{prop}

\begin{proof}
    It suffices to check for trivial bundles on a path connected space. One checks readily that $\ell(X)\left((\xi,n)\otimes (\xi',m)\right)\cong \ell(X)(\xi,n)\otimes\ell(X)(\xi',m)$. If $n=m$, so $n+m=0$, then we obtain the tensor product line bundle $\xi\otimes\xi'$ in degree 0 and the 0 line bundle in degree 1, with the zero differential between them. If $n\neq m$ then we have the reverse situation. The definition of the morphisms in $\FHT(X)$ makes it clear that their tensor product is similarly preserved by $\ell$. 
\end{proof}

\begin{thm}\label{thm:secondkinvariantofpicKU}
When $\mb F=\mb C$ the connective spectrum $\LR\FHTC(\ast)$ fits into a cofiber sequence $\Sigma^2H\mathbb{Z}\rightarrow \LR\FHTC(\ast)\rightarrow H\mathbb{Z}/2$ and its $k$-invariant is nontrivial.
\end{thm}
\begin{proof}
Consider the simplicial colimit of Definition \ref{defn:Rinvariant} used to define $\LR\FHTC(\ast)$. The connective spectra $\FHTC(\Delta^k)$ appearing in this colimit have exactly two nonzero homotopy groups, i.e.~they are stable 2-types in the sense of \cite{johnsonosornoPicard} (cf.~the discussion preceding Remark \ref{lem:nontriviality of first pics k invariant}). These homotopy groups are determined by the group of isomorphism classes of objects of $\FHTC(\Delta^k)$ and by the automorphisms of any one of those objects. In this case we have that  $\pi_0(\FHTC(\Delta^k))\cong\mathbb{Z}/2$ since $\Delta^k$ is connected and $\pi_1(\FHTC(\Delta^k))=\text{Top}(\Delta^k,\CC^\times)$ (since that is the group of automorphisms of the trivial bundle on $\Delta^k$), where $\CC^\times$ is considered as a topological group with the usual subspace topology inherited from $\CC$. 

Again as a result of the analysis of \cite{johnsonosornoPicard}, we get that the $k$-invariant is determined by the element of $H^2(H\mathbb{Z}/2;\text{Top}(\Delta^k,\CC^\times))\simeq \text{Hom}_\text{Ab}(\mathbb{Z}/2,\text{Top}(\Delta^k,\CC^\times))$ corresponding to the function that assigns to an element in each isomorphism class the symmetry isomorphism of its tensor square. In this case this is the function that sends $n$ to the constant function on $\Delta^k$ with value $(-1)^{n^2}$. 

The ring $\ZZ/2$ can be thought of as a Picard groupoid with $\pi_0\cong \ZZ/2$ and $\pi_1=0$. Therefore there is a forgetful functor of Picard groupoids $(\xi,n)\mapsto n$, corresponding to a map of spectra $\FHTC(\Delta^k)\to H\ZZ/2$ whose fiber is the Eilenberg-MacLane spectrum of $\pi_1(\FHTC(\Delta^k))$. Therefore there is a cofiber sequence of simplicial spectra
\[\Sigma H\text{Top}(\Delta^\bullet,\mathbb{C}^\times)\rightarrow \FHTC(\Delta^\bullet)\rightarrow H\mathbb{Z}/2\rightarrow\Sigma^2 H\text{Top}(\Delta^\bullet,\mathbb{C}^\times)\] in which the third term is a constant simplicial diagram and the last map is (the amalgam of) the $k$-invariants just described. 

Note that the leftmost term is equivalent to $\Sigma H \mathrm{Sing}_\bullet(\CC^\times)$, the levelwise suspension of the Eilenberg-MacLane spectrum of the singular simplicial group of $\CC^\times$. Because $\Omega^\infty$ preserves geometric realizations we can compute the homotopy groups of this colimit in $\Spaces$, which are trivial except in $\pi_1$, where they are $\CC^\times$ (with the discrete topology). This implies that the colimit of the above cofiber sequence is the cofiber sequence $\Sigma  H\mathbb{C}^\times \rightarrow \LR\FHT(\ast)\rightarrow H\mathbb{Z}/2\rightarrow\Sigma^2H\mathbb{C}^\times$ with the last map being the element of $H^2(H\mathbb{Z}/2;\mathbb{C}^\times)=\text{Hom}_\text{Ab}(\mathbb{Z}/2,\mathbb{C}^\times)$ sending $n$ to $(-1)^{n^2}\in\mathbb{C}^\times$. Now $\mathbb{C}^\times\simeq B\ZZ$ so $H\mathbb{C}^\times\simeq \Sigma H\mathbb{Z}$. Thus the preceding cofiber sequence can be rewritten as $$\Sigma^2 H\mathbb{Z}\rightarrow \LR\FHTC(\ast)\rightarrow H\mathbb{Z}/2\rightarrow\Sigma^3H\mathbb{Z}$$ and the $k$-invariant is still nontrivial of course (and therefore represented by $\beta Sq^2$).
\end{proof}

\begin{rmk}
    Note that although $\FHTC(\Delta^k)$ has homotopy groups concentrated in degrees 0 and 1 for all $k$, its localization $\LR\FHTC(\ast)$ has homotopy groups concentrated in degrees 0 and 2. 
\end{rmk}

\begin{cor}\label{cor:gl1KUsplitting} When $\FF=\CC$ the spectrum $\gl_1(KU)$ splits as $$\LR\FHTC(\ast)\oplus\gl_1(KU)[3,\infty).$$
\end{cor}
\begin{proof}
After applying $\LR$, the functor of Proposition \ref{prop:FHTinCG} induces a morphism of sheaves of commutative connective ring spectra. By Corollary \ref{cor:gl1DSVC gives gl1KU}, we have a map of spectra
\[\LR\FHTC(\ast)\rightarrow \LR\gl_1\CGC(\ast)\simeq\gl_1(KU)\] after evaluating at the point. This map is an isomorphism on $\pi_0$ because $\pi_0$ of these spectra can be computed by looking only at the bottom two levels of the simplicial diagram defining $\LR$ (i.e.~the coequalizer diagrams). It is of course an isomorphism on $\pi_1$ because both spectra have trivial homotopy in degree 1. Even further, it is an isomorphism on $\pi_2$. To see this, first note that whenever $\mathcal{F}$ is a concordance invariant sheaf there is an equivalence (as in the proof of Theorem \ref{thm:LRDSV is ku}) $\mathcal{F}(X)\simeq Map(\Sigma^\infty_+X,\mathcal{F}(\ast))$. Therefore the claimed isomorphism on $\pi_2$ would follow from an isomorphism $\pi_0\LR\FHTC(S^2)\simeq \pi_0\LR\gl_1\!\CGC(S^2)$. Since the inclusion of spaces induces an isomorphism $\pi_0\LR\gl_1\!\CGC(S^2)\rightarrow \pi_0\LR\CGC(S^2)$ and the latter group is $ku^0(S^2)$, the desired claim follows from the fact that both generators of $ku^0(S^2)$ are represented by line bundles and are therefore in the image of $\LR\FHTC$. Therefore the composite
\[\LR\FHTC(\ast)\rightarrow \gl_1(KU)\rightarrow\gl_1(\KU)[0,2]\]
is an equivalence, and the lemma follows.
\end{proof}

\begin{cor}\label{cor:gl1KU has nontrivial first k invariant}
The first $k$-invariant of $\gl_1(\KU)$ is non-trivial.
\end{cor}

\noindent Similar arguments prove the analogous statements for $\LR\mathscr{L}^\RR$:

\begin{thm}\label{thm:secondkinvariantofpicKO}
When $\mb F=\RR$ the connective spectrum $\LR\FHTR(\ast)$ fits into a cofiber sequence $\Sigma H\ZZ/2\rightarrow \LR\FHTR(\ast)\rightarrow H\mathbb{Z}/2$ and its $k$-invariant is nontrivial.
\end{thm}

\begin{cor}\label{cor:gl1KOsplitting} 
When $\FF=\RR$ the spectrum $\gl_1(\KO)$ splits as $$\LR\FHTR(\ast)\oplus\gl_1(\KO)[2,\infty).$$
\end{cor}

\begin{cor}\label{cor:gl1KO has nontrivial first k invariant}
The first $k$-invariant of $\gl_1(\KO)$ is non-trivial.
\end{cor}

\begin{conj}\label{rmk:k invariants of K of discrete fields}
Let $\FF$ be a discrete field. Then there are equivalences $\LR\CG(\ast)\simeq K(\FF)$ and $\LR\gl_1\!\CG(\ast)\simeq\gl_1(K(\FF))$. Moreover, there is a splitting $\LR\FHT(\ast)$ is a split summand of $\gl_1K(\FF)$ and the first and second $k$-invariants of $\pic(K(\FF))$ are non-trivial.
\end{conj}

\section{Computations of $k$-invariants}

In this section we compute the possible $\EE_\infty$ structures on $h$-types with the same homotopy groups as $\Pic_0^1(\KO)$, $\Pic_0^1(\KU)$, $\GBRconn$, and $\GBCconn$. We do this by computing the possible $k$-invariants of spectra with the same homotopy groups. It will be useful to recall that $\GBRconn$ and $\GBCconn$  are equivalent to $\BGL_1(\KO)[0,2]$ and $\BGL_1(\KU)[0,3]$ respectively. In most cases we can explicitly name these $k$-invariants in terms of cohomology operations.

\begin{prop}\label{prop:spacelevelsplittingsofPicKOKU}
There are equivalences of $h$-types \[\GBR\simeq \ZZ/8\times K(\ZZ/2,1)\times K(\ZZ/2,2)\] and \[\GBC\simeq \ZZ/2\times K(\ZZ/2,1)\times K(\ZZ,3).\]
\end{prop}

\begin{proof}
The first $k$-invariant of an $\EE_\infty$-type is always null, by \cite{arlettaz-firstkinvariant}.  From \cite[Lemma 3.1]{mayRingSpacesRingSpectra} we have that the $1$-component of $\Pic(KO)$ splits as $K(\ZZ/2,1)\times \BSO$. The fact that $\BSO[0,2]\simeq K(\ZZ/2,2)$ completes the proof. The case of $\Pic(KU)$ is essentially identical (the result of \cite{mayRingSpacesRingSpectra} applies to both $KO$ and $KU$).
\end{proof}

\subsection{First $k$-invariants}

\noindent Now we determine the possible first $k$-invariants of the associated spectra $\gbr$ and $\gbc$. 

\begin{lem}\label{lem:firstkinvariantpossibilities}
The first $k$-invariant of $\pic(KO)$ is either trivial or $\Sq^2\circ\rho\colon H\ZZ/8\to\Sigma^2H\ZZ/2$, where $\rho\colon H\ZZ/8\to H\ZZ/2$ is the reduction mod 2 map, and the first $k$-invariant of $\pic(KU)$ is either trivial or $\Sq^2\colon H\ZZ/2\to \Sigma^2\ZZ/2$. 
\end{lem}

\begin{proof}
The case for $KU$ is immediate because $H^2(H\ZZ/2;\ZZ/2)\cong\ZZ/2$. The case of $KO$ follows from the group cohomology computation $H^2(H\ZZ/8;\ZZ/2)\cong H^2_{gp}(\ZZ/8,\ZZ/2)\cong\ZZ/2$.
\end{proof}

\begin{cor}
The $k$-invariants of $\pic_0^1(KO)$ and $\pic_0^1(KU)$ are $\Sq^2\circ\rho$ and $\Sq^2$ respectively, where $\rho\colon\ZZ/8\to\ZZ/2$ is the reduction mod 2 map. Therefore $\Sq^2\circ\rho$ and $\Sq^2$ are also the first $k$-invariants of $\gbr$ and $\gbc$. 
\end{cor}

\begin{proof}
This follows from Lemmas \ref{lem:nontriviality of first pics k invariant}, \ref{lem:nontrivialityoffirstk} and \ref{lem:firstkinvariantpossibilities}.
\end{proof}

\noindent Next we wish to determine the \textit{second} $k$-invariants of $\gbc$ and $\gbr$. We begin by determining the $k$-invariants of their connected covers $\bgl_1(\KO)[0,2]$ and $\bgl_1(\KU)[0,3]$. This is of course equivalent to determining the $k$-invariants of $\gl_1(\KO)[0,1]$ and $\gl_1(\KU)[0,2]$. The first is almost trivial, and we prove it in Proposition \ref{lem:possiblekinvariantsofbgl1KO}. For the second case, more work is required, and we first prove several lemmas.

\begin{prop}\label{lem:possiblekinvariantsofbgl1KO}
The $k$-invariant of $\gbrconn$ is $\Sq^2\colon \Sigma H\ZZ/2\to \Sigma^3 H\ZZ/2$. Equivalently, there are two $\EE_\infty$ structures on $K(\ZZ/2,1)\times K(\ZZ/2,2)$ and the one on $\BGL_1(\KO)[0,2]$ is the one which is not the product structure.
\end{prop}

\begin{proof}
By considering the Postnikov tower and knowing that $H^1(H\ZZ/2;\ZZ/2)\cong\ZZ/2$ is generated by $\Sq^2$ we see that the only two possible $k$-invariants are $0$ and $\Sq^2$. The result then follows from Theorem \ref{thm:secondkinvariantofpicKO} and Corollary \ref{cor:gl1KOsplitting}.
\end{proof}

Next we determine the $k$-invariant of $\gbcconn$. We begin by determining all possible $k$-invariants of a spectrum with the same homotopy groups, or equivalently, all possible infinite loop space structures on the space $\Omega^\infty\gbcconn\simeq\BGL_1(\KU)[0,3]$. 

\begin{lem}\label{lem:AandB}
There are exactly two $h$-types whose only nontrivial homotopy groups are $\pi_1\cong\ZZ/2$ and $\pi_3\cong\ZZ$.
\end{lem}
\begin{proof}
The result follows immediately from the computation $H^4(B\mathbb{Z}/2;\mathbb{Z})\cong\mathbb{Z}/2$.
\end{proof}

\begin{lem}\label{noinfB}
Let $X$ be the $h$-type with $\pi_1(X)\cong\ZZ/2$ and $\pi_3(X)\cong\ZZ$ and non-trivial $k$-invariant. Then $X$ does not admit an $\EE_\infty$-structure.  
\end{lem}
\begin{proof}
If $X$ admitted an $\EE_\infty$-structure then its $k$-invariant $K(\ZZ/2,1)\to K(\ZZ,4)$ would be $\Omega^\infty$ of a stable $k$-invariant $\Sigma H\ZZ/2\to\Sigma^4H\ZZ$ and would therefore induce a map of Abelian groups on cohomology. Let $\gamma\in H^4(K(\ZZ/2,1);\mathbb{Z})\cong\ZZ/2$ be the $k$-invariant of $X$ and let $\alpha\in H^2(K(\ZZ/2,1);\mathbb{Z})$ be the generator given by the inclusion $B\mathbb{Z}/2\rightarrow BU(1)=B^2\mathbb{Z}$. Then $\gamma$ must be the cup-square of $\alpha$.

Now for any $h$-type $Y$ the cohomology operation induced by $\gamma$ is the map
\begin{align*}
H^1(Y;\mathbb{Z}/2)&\rightarrow H^4(Y;\mathbb{Z})\\
a&\mapsto \beta(a)^2
\end{align*}
 where $\beta$ is the Bockstein map. Let $Y=K(\ZZ/2,1)\times K(\ZZ/2,1)$. If we take $a$ and $b$ to be the generators of $H^1(Y;\mathbb{Z}/2)$ then we see that the cross term in $\beta(a+b)^2=\beta(a)^2+2\beta(a)\beta(b)+\beta(b)^2$ is non-zero, and therefore the above map is not an Abelian group homomorphism. In other words, the $k$-invariant cannot be $\Omega^\infty$ of a stable $k$-invariant.
\end{proof}

\begin{prop}\label{lem:stablecalc}
For any $n$, $H^{n+3}(\Sigma^nH\mathbb{Z}/2;\ZZ)$ is isomorphic to $\mathbb{Z}/2$, generated by  $\Sigma^n(\beta\circ Sq^2):\Sigma^nH\mathbb{Z}/2\rightarrow\Sigma^{n+1}H\mathbb{Z}/2\rightarrow\Sigma^{n+3}H\mathbb{Z}$, the appropriate suspension of the composite of the Bockstein and $Sq^2$.
\end{prop}
\begin{proof}
It suffices to calculate $H^4(\Sigma H\ZZ/2;\ZZ)\simeq H^3(H\ZZ/2,\ZZ)$. Note that a map $f\co X\rightarrow \Sigma^3H\ZZ$ factors through the Bockstein $\Sigma^3\beta\co\Sigma^{2}H\ZZ/2\rightarrow\Sigma^3H\ZZ$ if and only if the composite of $f$ with the map $\Sigma^3 2\co\Sigma^3H\ZZ\rightarrow\Sigma^3H\ZZ$ is null, since $\beta$ is the fiber of multiplication by 2. 
But the composite of any map $\alpha\co H\mathbb{Z}/2\rightarrow\Sigma^3H\mathbb{Z}$ with $\Sigma^3 2\co\Sigma^3H\mathbb{Z}\rightarrow\Sigma^3H\mathbb{Z}$ is null because $H^3(H\ZZ/2;\ZZ)$ is 2-torsion. 
So every element $\alpha\in H^3(H\mathbb{Z}/2;\ZZ)$ factors as $\beta\alpha'$ for some $\alpha'\in H^2(H\mathbb{Z}/2;\ZZ/2)$. Therefore $\beta_\ast\colon \ZZ/2\cong H^2(H\ZZ/2;\ZZ/2)\to H^3(H\ZZ/2;\ZZ)$ is a surjection and $H^3(H\mathbb{Z}/2;\ZZ)$ has only two elements, $0$ and $\beta\circ\Sq^2$. Therefore it only remains to check that $\beta\circ\Sq^2$ is non-zero. 

Suppose that $\beta\circ\Sq^2$ were null. Then $Sq^2$ would lift to a map $H\mathbb{Z}/2\rightarrow\Sigma^2H\mathbb{Z}$, i.e.~there would be a factorization through the fiber of $\Sigma^3\beta$
\[
\begin{tikzcd}
&\Sigma^2H\ZZ/2\ar[d,"\Sq^2"]\ar[dl,"\omega",swap]&\\
\Sigma^2H\ZZ\ar[r,"\Sigma^2\rho"] & \Sigma^2H\ZZ\ar[r,"\Sigma^3\beta"] & \Sigma^3H\ZZ
\end{tikzcd}
\]
where $\rho$ is reduction mod 2. But, similarly to above, every class in $H^2(H\ZZ/2;H\ZZ)$ is 2-torsion and therefore there is another factorization 
\[
\begin{tikzcd}
H\ZZ/2\ar[r,"\omega"]\ar[d,"\Sq^1",swap] & \Sigma^2H\ZZ\\
\Sigma H\ZZ/2\ar[ur,"\beta",swap]
\end{tikzcd}
\]

\noindent Therefore we have a commutative diagram
\[
\begin{tikzcd}
H\ZZ/2\ar[r,"\Sq^2"]\ar[d,"\beta\circ\Sq^1",swap] & \Sigma^2H\ZZ/2\\
\Sigma^2 H\ZZ\ar[ur,"\Sigma^2\rho",swap]
\end{tikzcd}
\]

\noindent This however is a contradiction because the composite $\Sigma^2\rho\circ\beta\circ\Sq^1$ must be trivial on cohomology classes in degree greater than 1, whereas $\Sq^2$ is not. Therefore $H^3(H\ZZ/2;\ZZ)$ is isomorphic to $\ZZ/2$ and is generated by $\beta\circ\Sq^2$. 
\end{proof}

\begin{prop}\label{cor:options for gl1KU k invariant}
Any $h$-type equivalent to $K(\ZZ/2,1)\times K(\ZZ,3)$ admits exactly two $\EE_\infty$-structures. One of them is the product structure and the other is the one associated to the stable $k$-invariant $\Sigma(\beta\circ\Sq^2)\co \Sigma H\ZZ/2\to\Sigma^4 H\ZZ$.  
\end{prop}

\begin{proof}
By Proposition \ref{lem:stablecalc} there are at most two $\EE_\infty$-structures on $K(\ZZ/2,1)\times K(\ZZ,3)$, one associated to the coproduct spectrum $\Sigma H\ZZ/2\vee \Sigma^3H\ZZ$ and the other associated to the spectrum with the same homotopy groups but non-trivial $k$-invariant $\Sigma(\beta\circ\Sq^2)\co \Sigma H\ZZ/2\to\Sigma^4H\ZZ$. By taking $\Omega^\infty$, the latter yields an infinite loop space, distinct from the product infinite loop space, with homotopy groups $\pi_1\cong \ZZ/2$ and $\pi_3\cong\ZZ$. By Lemma \ref{lem:AandB} this space has either the trivial or non-trivial $k$-invariant, and by Lemma \ref{noinfB} it cannot have the non-trivial $k$-invariant. Thus each element of $H^4(\Sigma H\ZZ/2;\ZZ)$ induces a distinct infinite loop space structure on the product $K(\ZZ/2,1)\times K(\ZZ,3)$.
\end{proof}

\begin{cor}\label{cor:kinvariantoftruncatedgl1}
    The $k$-invariant of $\gbcconn$ is $\beta\circ\Sq^2$ and therefore the infinite loop space structure on $\GBCconn\simeq K(\ZZ/2,1)\times K(\ZZ,3)$ is not the product structure.
\end{cor}

\begin{proof}
    This follows from Theorem \ref{thm:secondkinvariantofpicKU}, Corollary \ref{cor:gl1KUsplitting} and Proposition \ref{cor:options for gl1KU k invariant}.
\end{proof}

The following proposition is not immediately relevant but will be used later and follows naturally from Corollary \ref{cor:kinvariantoftruncatedgl1}. We do not know if the splitting also exists at the level of $\EE_\infty$-types. 

\begin{prop}\label{prop:loopspacesplittingofGBC}
    There is a splitting of $\EE_1$-types $\GBC\simeq \ZZ/2\times \GBCconn$.
\end{prop}

\begin{proof}
Consider the cofiber sequence of spectra 
\[
\gbcconn\to\gbc\to H\ZZ/2.
\]
The proposition is equivalent to the claim that the connecting map $H\ZZ/2\to\Sigma\gbcconn$ is null after applying $\Omega^\infty\Sigma(-)$. Consider the second cofiber sequence

\[
\Sigma^2H\ZZ/2\xrightarrow{\beta \Sq^2}\Sigma^5H\ZZ\to\Sigma^2\gbcconn\to\Sigma^3H\ZZ/2\xrightarrow{\beta\Sq^2}\Sigma^6H\ZZ
\]
Applying $[\Sigma H\ZZ/2,-]$ produces an exact sequence:
\begin{align*}
H^1(H\ZZ/2;\ZZ/2)&\xrightarrow{\beta\Sq^2}H^4(H\ZZ/2;\ZZ)\to [\Sigma H\ZZ/2,\Sigma^2\gbcconn]\\
&\to H^2(H\ZZ/2;\ZZ/2)\xrightarrow{\beta\Sq^2}H^5(H\ZZ/2;\ZZ)
\end{align*}
By computing the relevant cohomology groups this can be rewritten as:
\begin{align*}
\ZZ/2\{\Sq^1\}\xrightarrow{\beta\Sq^2}\ZZ/2\{\beta\Sq^2\Sq^1\}&\to [\Sigma H\ZZ/2,\Sigma^2\gbcconn]\\&\to \ZZ/2\{\Sq^2\}\xrightarrow{\beta\Sq^2}\ZZ/2\{\beta\Sq^4\}
\end{align*}
From this it immediately follows that the first map in the sequence is surjective, and the final map is zero, so the third map gives an isomorphism \[[\Sigma H\ZZ/2,\Sigma^2\gbcconn]\cong [\Sigma H\ZZ/2,\Sigma^3H\ZZ/2]\cong\ZZ/2\{\Sq^2\}
\] 
Now recall that $\Sq^2\colon \Sigma H\ZZ/2\to\Sigma^3H\ZZ/2$ induces the null map on underlying spaces $B\ZZ/2\to B^3\ZZ/2$. Therefore, for both infinite-loop maps $B\ZZ/2\to B^2\GBCconn$, the composite $B\ZZ/2\to B^2\GBCconn\to B^3\ZZ/2$ is null. It follows that any such map factors as $B\ZZ/2\to B^5\ZZ\to B^2\GBCconn$. But $H^5(B\ZZ/2;\ZZ)=0$.

\end{proof}

\subsection{Extending to $\gbr$ and $\gbc$}

So far we have computed the $k$-invariants of $\pic_0^1(\KO)$, $\pic_0^1(\KU)$, $\gbrconn$ and $\gbcconn$. Now we wish to determine how this data can be glued together to understand the $k$-invariants of $\gbr$ and $\gbc$. Our computations only determine the second $k$-invariants of these spectra up to isomorphism. In other words, for $\KO$ we determine that the second $k$-invariant of $\gbr$ is one of two elements of $H^3(\pic_0^1(\KO);\ZZ)\cong\ZZ/2\times\ZZ/2$, and for $\KU$ it is one of the two generators of $H^4(\pic_0^1(\KU))\cong\ZZ/4$. In the latter case there is an equivalence of Postnikov towers which interchanges the two generators, but in the former case it is less clear which generator one should choose. Ultimately, however, the choice will not matter because both $k$-invariants work for the applications of Section \ref{sec:interpretations}. In particular, both choices satisfy the conclusions of Propositions \ref{prop:Maycock group structure from k invariant} and \ref{prop:Real Maycock group structure from k invariant}.

The proof of Lemma \ref{lem:Z4computationforKU}  was sketched for us by Tyler Lawson. Any mistakes are of course due to our own misunderstanding.

\begin{lem}\label{lem:SESforH4pic01}
    There is an exact sequence \[0\to H^4(H\ZZ/2;\ZZ)\to H^4(\pic_0^1(\KU);\ZZ)\to H^3(H\ZZ/2;\ZZ)\to 0\] in which the first group is generated by $\beta\Sq^2\Sq^1$ and the last group is generated by $\beta\Sq^2$.
\end{lem}

\begin{proof}
The generators of the first and last groups are standard computations. The fact that $H^4(H\ZZ/2;\ZZ)\to H^4(\pic_0^1(\KU);\ZZ)\to H^3(H\ZZ/2;\ZZ)$ is exact follows immediately from considering the applying $H^4(-;\ZZ)$ to the cofiber sequence $\Sigma H\ZZ/2\to \pic_0^1(\KU)\to H\ZZ/2$ from the Postnikov tower of $\pic(\KU)$. The fact that the entire sequence is short exact follows on the left from the fact that $H^2(H\ZZ/2;\ZZ)=0$. For the right hand side, first recall that the first $k$-invariant of $\pic(\KU)$, i.e.~the only $k$-invariant of $\pic_0^1(\KU)$, is $\Sq^2\colon H\ZZ/2\to\Sigma^2H\ZZ/2$. Therefore there is an exact sequence \[H^4(\pic_0^1(\KU);\ZZ)\to H^3(H\ZZ/2;\ZZ)\to H^5(H\ZZ/2;\ZZ)\] in which the last map is $\Sq^2$. Since $H^3(H\ZZ/2;\ZZ)\cong\ZZ/2$ is generated by $\beta\Sq^2$ the image of this map is $\beta\Sq^2\Sq^2$. Taking the quotient by 2, which is an injection $H^\ast(H\ZZ/2;\ZZ)\to H^\ast(H\ZZ/2;\ZZ/2)$, we get $\Sq^1\Sq^2\Sq^2$, which is zero by the Adem relations. Therefore the image of $\Sq^2\colon H^3(H\ZZ/2;\ZZ)\xrightarrow{}H^5(H\ZZ/2;\ZZ)$ is zero.
\end{proof}

\begin{lem}\label{lem:Z4computationforKU}
   There is an isomorphism $H^4(\pic_0^1(\KU);\ZZ)\cong\ZZ/4$. 
\end{lem}

\begin{proof}
Consider the cofiber sequence $\sph\xrightarrow{2}\sph\xrightarrow{q}\sph/2$, where $\sph/2$ is the mod-2 Moore spectrum. Let $\overline{f}\colon\sph\to\pic_0^1(\KU)$ be the generator of $\pi_0(\pic_0^1(\KU))\cong\ZZ/2$. Then since $\pi_0(\pic_0^1(\KU))$ is 2-torsion, $\overline{f}$ lifts to a non-zero map $f\colon \sph/2\to \pic_0^1(\KU)$. The long exact sequence in homotopy applied to the above cofiber sequence shows that $\pi_0(\sph/2)\cong\ZZ/2$, generated by $q$ and  $\pi_1(\sph/2)\cong\ZZ/2$, generated by $q\circ\eta$. By considering the commutative diagram below, in which $d=0,1$ and $\phi=id,\eta$ respectively, we have that $f$ must be non-zero on $\pi_0$ and $\pi_1$ and therefore an isomorphism on those two homotopy groups.

\[
\begin{tikzcd}
 \Sigma^d\sph\ar[d,"\phi",swap]\ar[dr,"q\circ\phi"]&\\
\sph\ar[r,"q"]\ar[d,"\overline{f}",swap] & \sph/2\ar[dl,"f"]\\
\mathrm{pic}_0^1(\KU) & 
\end{tikzcd}
\]

\noindent As a result, the cofiber of $f$, $\cofib(f)$, is 2-connected. Applying the Hurewicz theorem to $\cofib(f)$, we find that $\pi_3(\cofib(f))\cong H^3(\cofib(f);\ZZ)$. Applying the long exact sequence in homotopy to the cofiber sequence $\sph/2\to \pic_0^1(\KU)\to\cofib(f)$ we see that $\pi_3(\cofib(f))\cong\pi_2(\sph/2)$. Now applying the long exact sequence for homology to the same cofiber sequence, we get $\pi_2(\sph/2)\cong H^3(\pic_0^1(\KU);\ZZ)$. The universal coefficient theorem tells us that $H^4(\pic_0^1(\KU);\ZZ)\cong \Ext(\pi_2(\sph/2),\ZZ)\cong \pi_2(\sph/2)^\vee\cong \pi_2 (\sph/2)$. We conclude by pointing out that $\pi_2(\sph/2)\cong\ZZ/4$. The authors cannot find this final fact in the published literature, but several sketch proofs of it are provided in \cite{Mathew-MO_question_moore_spectrum}.
\end{proof}

\begin{prop}\label{prop:KU k invariant generates Z4}
    The second $k$-invariant of $\gbc$ generates the cohomology group $H^4(\pic_0^1(\KU);\ZZ)\cong\ZZ/4$.
\end{prop}

\begin{proof}
    By Lemma \ref{lem:Z4computationforKU} and Theorem \ref{thm:secondkinvariantofpicKU} there are three options for the second $k$-invariant of $\gbc$, any of the non-trivial maps $\pic_0^1(\KU)\to\Sigma^4H\ZZ$. However, by Corollary \ref{cor:kinvariantoftruncatedgl1}, the actual second $k$-invariant must give $\beta\Sq^2$ on $\gbcconn$ after taking a connected cover. This corresponds to asking for maps $\pic_0^1(\KU)\to\Sigma^4 H\ZZ$ which restrict to $\beta\Sq^2$ when precomposing with the map $\Sigma H\ZZ/2\to\pic_0^1(\KU)$ in the Postnikov tower of $\pic(\KU)$. But by Lemma \ref{lem:SESforH4pic01}, the map $\ZZ/4\cong H^4(\pic_0^1(\KU);\ZZ)\to H^3(\ZZ/2;\ZZ)\cong\ZZ/2$ is a surjection, and the codomain is generated by $\beta\Sq^2$. Therefore both of the generators of $H^4(\pic_0^1(\KU);\ZZ)$ satisfy the necessary property, so one of them of must be the $k$-invariant of $\gbc$. 
\end{proof}

\begin{thm}\label{thm:both k invariants work for picKU}
    The two generators of $\ZZ/4\cong H^4(\pic_0^1(\KU);\ZZ)$ yield equivalent Postnikov sections, and hence both present $\gbc$. 
\end{thm}

\begin{proof}
    
    Let $a$ and $b$ be the two generators of $\ZZ/4\cong H^4(\pic_0^1(\KU);\ZZ)$. Since $a=-b$ there is a commutative diagram
    \[
\begin{tikzcd}
\mathrm{pic}_0^1(\KU)\ar[r,"a",swap]\ar[d,"id"]&\Sigma^4H\ZZ\ar[d,"-1"]\\
\mathrm{pic}_0^1(\KU)\ar[r,"b"] & \Sigma^4H\ZZ
\end{tikzcd}.
\]
The induced map between the fibers of the horizontal maps is an equivalence between the Postnikov sections corresponding to the $k$-invariants $a$ and $b$.
\end{proof}

\begin{rmk}
    While we do not have a geometric argument at hand, it seems almost certain that the automorphism used in the proof of Theorem \ref{thm:both k invariants work for picKU} corresponds to the complex conjugation automorphism on $\KU$.
\end{rmk}

\begin{prop}\label{prop:KO k invariant ambiguity}
    There is an isomorphism $H^3(\pic_0^1(\KO);\ZZ/2)\cong\ZZ/2\times\ZZ/2$. Moreover, pulling back along the fiber in the Postnikov tower for $\pic_0^1(\KO)$, $\Sigma H\ZZ/2\to \pic_0^1(\KO)$, induces a surjection $\ZZ/2\times\ZZ/2\to\ZZ/2\{Sq^2\}$.
\end{prop}

\begin{proof}
    For readability do not include the entire proof. We only note that it arises from considering the long exact sequence in mod 2 cohomology applied to the Postnikov tower $\Sigma H\ZZ/2\to \pic_0^1(\KO)\to H\ZZ/8$ and a large number of low degree cohomology computations for the Eilenberg-MacLane spectra $H\ZZ/8$ and $H\ZZ/2$. 
\end{proof}

\begin{rmk}
    For the time being, we do not know how specify the ``correct'' $k$-invariant for $\gbr$, as it could be one of two elements in the primage of $Sq^2$. In the case of $\pic(\KU)$ the ambiguity is irrelevant up to equivalence (cf.~Theorem \ref{thm:both k invariants work for picKU}), but a similar approach will not work here. It may be possible to resolve the ambiguity by taking homotopy fixed points of $\pic(\KU)$ and comparing the second $k$-invariant of the resulting fixed point spectrum, via the homotopy fixed points spectral sequence, to the two possibilities given in Proposition \ref{prop:KO k invariant ambiguity}. Luckily, this uncertainty does not effect the group structure on the $\gbr$-cohomology of a space.
\end{rmk}

\subsection{Group structures}

Now that we know the $k$-invariants of $\gbr$ and $\gbc$, we can determine the group laws for $\gbr^0(X)$ and $\gbc^0(X)$, which we will use in the next section.

\begin{prop}\label{prop:Maycock group structure from k invariant}
For a space $X$, the group law on the set $\gbc^0(X)\cong H^0(X;\ZZ/2)\times H^1(X;\ZZ/2)\times H^3(X;\ZZ)$ is given by $(a,b,c)\boxplus (a',b',c')=(a+a',b+b', c+c'+\beta(b\cup b'))$, where we abuse notation and use the symbol $+$ to denote the usual addition in $H^0(X;\ZZ/2)$, $H^1(X;\ZZ/2)$ and $H^3(X;\ZZ)$.
\end{prop}

\begin{proof}
Proposition \ref{prop:loopspacesplittingofGBC} implies that the first coordinate of the group splits off. Therefore it suffices to prove that the group structure on $\gbcconn^0(X)$ is $(b,c)+(b',c')=(b+b', c+c'+\beta(b\cup b'))$. Consider the natural short exact sequence of Abelian groups $H^3(X;\ZZ)\rightarrow\gbcconn^0(X)\rightarrow H^1(X;\ZZ/2)$.
The cocycle for that group extension is a natural map $H^1(X;\ZZ/2)\times H^1(X;\ZZ/2)\rightarrow H^3(X;\ZZ)$, which is represented by some map $K(\ZZ/2,1)\times K(\ZZ/2,1)\rightarrow K(\ZZ,3)$. There are only two such maps, the trivial one and $\beta(-\cup -)$. The cocycle cannot be trivial, otherwise $\gbcconn$ would be a direct sum of Eilenberg-Maclane spectra, which contradicts Corollary \ref{cor:kinvariantoftruncatedgl1}.
\end{proof}

\noindent The proof of the following proposition is similar.

\begin{prop}\label{prop:Real Maycock group structure from k invariant}
For an $h$-type $X$, the group law on the set $\gbr^0(X)\cong H^0(X;\ZZ/2)\times H^1(X;\ZZ/2)\times H^2(X;\ZZ/2)$ is given by $(a,b,c)\boxplus (a',b',c')=(a+a',b+b',c+c'+b\cup b')$, where we abuse notation and use the symbol $+$ to denote the usual addition in $H^0(X;\ZZ/8)$, $H^1(X;\ZZ/2)$ and $H^2(X;\ZZ/2)$.
\end{prop}

\section{What $\gbr$ and $\gbc$ Represent}\label{sec:interpretations}

In the following section we describe several algebraic and geometric interpretations of the cohomology theories associated to $\gbr$ and $\gbc$.

\subsection{Brauer Groups}

By determining the group structures of $\gbc^0(X)$ and $\gbr^0(X)$ for a space $X$ we now have that $\GBC$ and $\GBR$ represent well-known classical Brauer groups whose elements are Morita classes of: bundles of $\ZZ/2$-graded central simple algebras; bundles of $\ZZ/2$-graded continuous trace $C^\ast$-algebras; and bundles of super (i.e.~$\ZZ/2$-graded) 2-lines. All of these data were previously known to be isomorphic (for instance via graded versions of Dixmier-Douady theory) but using $\GBC$ and $\GBR$ to classify these data is new. This is consistent with the fact that all of these data can been used to twist $K$-theory.

\begin{defn}
Let $\mathrm{GBr\OO}(X)$ denote the Brauer group of (possibly infinite dimensional) graded, continuous trace, complex $C^\ast$-algebras with spectrum $X$, as in \cite{maycockparker}. Let $\mathrm{GBr\UU}(X)$ denote the Brauer group of (possibly infinite dimensional) graded, continuous trace, real $C^\ast$-algebras with spectrum $X$, as in \cite{moutuou-gradedBrauergroups}. 
\end{defn}

\noindent The following theorems are proven in \cite{moutuou-gradedBrauergroups,maycockparker}:

\begin{thm}[Maycock]
If $X$ is a space homotopy equivalent to a CW-complex then $\mathrm{GBr\UU}(X)\cong H^0(X;\ZZ/2)\times H^1(X;\ZZ/2)\times H^3(X;\ZZ)$ with group law $(a,b,c)+(a',b',c')=(a+a',b+b',c+c'+\beta(b\cup b'))$ where $\beta$ is the Bockstein homomorphism. 
\end{thm}

\begin{rmk}
    In \cite{maycockparker} the $H^0$ term is mostly ignored, since Maycock requires that her bundles have isomorphic fibers over every connected component. This assumption is unnecessary though, as shown in \cite{moutuou-gradedBrauergroups}.
\end{rmk}

\begin{thm}[Moutuou]
If $X$ is a space homotopy equivalent to a CW-complex then $\mathrm{GBr\OO}(X)\cong H^0(X;\ZZ/8)\times H^1(X;\ZZ/2)\times H^2(X;\ZZ/2)$ with group law $(a,b,c)+(a',b',c')=(a+a',b+b',c+c'+(b\cup b'))$. 
\end{thm}

Propositions \ref{prop:Maycock group structure from k invariant} and \ref{prop:Real Maycock group structure from k invariant} now imply the following corrolaries:

\begin{cor}\label{cor:GBrU is the same as gbc}
If $X$ is a space homotopy equivalent to a $CW$-complex then there is an isomorphism \[\mathrm{GBr\UU}(X)\cong \gbc^0(X).\]
\end{cor}

\begin{cor}\label{cor:GBrO is the same as gbc}
If $X$ is a space homotopy equivalent to a $CW$-complex then there is an isomorphism \[\mathrm{GBr\OO}(X)\cong \gbr^0(X).\]
\end{cor}

\begin{rmk}
    The Brauer group of $\ZZ/2$-graded continuous trace $C^\ast$-algebras with spectrum $X$ is equivalent to the Brauer group of $\ZZ/2$-equivariant $C^\ast$-algebras with spectrum $X$ with the property that the induced $\ZZ/2$-action on $X$ is trivial.  Our constructions of $\gbcconn$ and $\gbrconn$, and the fact that they represent these Brauer groups, should be compared to the construction of the cohomology theory $E_{\CC,\mb T}$ in \cite{evans-pennig--T_equivariant_DD_Theory}. It is shown therein that $E_{\CC,\mb T}^0(X)$ is the group of $\mathbb{T}$-equivariant line bundles on $X$ and that $E_{\CC,\mb T}^1(X)$ is the Brauer group of $\mb T$-equivariant $C^\ast$-algebras with spectrum $X$ where $X$ has trivial induced $\mb T$-action. Our constructions on the other hand show that $\gbcconn^{-1}(X)$ is the group of super line bundles on $X$ and that $\gbcconn^0(X)$ is the Brauer group of $\ZZ/2$-equivariant $C^*$-algebras with spectrum $X$ having trivial induced $\ZZ/2$-action. Moreover, $\Omega^\infty E_{\CC,\mb T}\simeq \ZZ\times K(\ZZ,2)$ and $\Omega^\infty \Sigma^{-1}\gbcconn\simeq \ZZ/2\times K(\ZZ,2)$. This suggests that Evans and Pennig's $E_{\CC,\mb T}$ spectrum is equivalent to a truncation of $\gl_1(KU_{\mathbb{T}})$, the space of units of $\mathbb{T}$-equivariant complex $K$-theory.
\end{rmk}

Propositions \ref{prop:Maycock group structure from k invariant} and \ref{prop:Real Maycock group structure from k invariant} also give $\gbc$ and $\gbr$ the following interpretations in terms of graded $2$-line bundles. 

 \begin{cor}\label{cor:picKU03 classifies super 2line bundles}     
Let $X$ be a smooth manifold. Then $\gbc^0(X)\cong sLBdl_{\CC}(X)$ and $\gbr^0(X)\cong sLBdl_{\RR}(X)$, where $sLBdl_{\CC}$ and $sLBdl_{\RR}$ are the Brauer groups of complex and real super 2-line bundles on $X$ as defined in \cite{Kristel-Ludewig-Waldorf--2vectorbundles}.
 \end{cor}

 \begin{proof}
     This follows immediately from either \cite[Theorem 4.4]{Kristel-Ludewig-Waldorf--2vectorbundles} or \cite[Theorem 2.2.6]{Mertsch-PhDThesis}.
 \end{proof}

Corollaries \ref{cor:GBrU is the same as gbc} and \ref{cor:picKU03 classifies super 2line bundles} have the following interpretation in terms of parameterized stable homotopy theory:
\begin{cor}\label{cor:brauer group is also ku bundles}
If $X$ is connected then there is an isomorphism between the Brauer group of graded, continuous trace, complex $C^\ast$-algebras with spectrum $X$, equivalently the Brauer group of complex super 2-line bundles on $X$, and the group of $\KU[0,2]$-line bundles on $X$, where the group structure on the latter is given by the parameterized tensor product. The same holds in the real case with $KU[0,2]$ replaced by $KO[0,1]$.
\end{cor}

\begin{proof}
We have a string of equivalences of infinite loop spaces $$\GBCconn\simeq \mathrm{B}\Omega\GBC\simeq\BGL_1(\KU)[0,3]\simeq \mathrm{B}(\GL_1(\KU)[0,2]).$$ Therefore it suffices to show that $\GL_1(\KU)[0,2]\simeq \GL_1(\KU[0,2])$, which follows from Lemma \ref{lem:truncation commutes with GL1} below. The argument for the real case is identical.
\end{proof}

\begin{rmk}
    Corollary \ref{cor:brauer group is also ku bundles} cannot be stated in terms of all of $\GBC$ for the following reason: the spectrum $\KU[0,2]$ is no longer 2-periodic, so $\pi_0(\Pic(\KU[0,2]))$ will be at least $\ZZ$. As a result we cannot think of maps $X\to\GBC$ as bundles of invertible $\KU[0,2]$-modules on $X$. 
\end{rmk}

\begin{lem}\label{lem:truncation commutes with GL1}
Let $R$ be a connective commutative ring spectrum and $n\in\ZZ$ a positive integer. Then there is an equivalence of infinite loop spaces \[\GL_1(R[0,n])\simeq\GL_1(R)[0,n].\]
\end{lem}

\begin{proof}
Consider the zig-zag of infinite loop maps $$\GL_1(R[0,n])\to \GL_1(R[0,n])[0,n]\leftarrow \GL_1(R)[0,n]$$ where the arrow from left to right is the usual map from a space to its truncation and the arrow from right to left is obtained by truncating after applying the functor $\GL_1$ to the ring map $R\to R[0,n]$. Recall that the homotopy groups of $\GL_1(R)$ are isomorphic to those of $R$ except in degree zero where $\pi_0(\GL_1(R))\cong\pi_0(R)^\times$. Therefore $GL_1(R[0,n])$ is an $n$-truncated space and the left to right map from itself to to its truncation is an equivalence.  Because $R\to R[0,n]$ is an equivalence through homotopy degree $n$, so is $\GL_1(R)\to\GL_1(R[0,n])$ and therefore the left to right map is also an equivalence. 
\end{proof}

\begin{rmk}
There is an equivalence of commutative ring spectra $\KU[0,2]\simeq\ku[0,2]$ which is the truncation of the equivalence $\ku\to\KU[0,\infty)$. This is essential to the use of Lemma \ref{lem:truncation commutes with GL1} in the proof of Corollary \ref{cor:brauer group is also ku bundles}. In particular we must construct $\KU[0,2]$ by first taking the connective cover of $\KU$ and then truncating, as $\KU(-\infty,2]$ is not even a ring spectrum. 
\end{rmk}

\begin{rmk}
    Recall that there is a $C_2$-action on $\KU$ by complex conjugation whose fixed point spectrum is $\KO$. The $C_2$-equivariant complex $K$-theory spectrum is often denoted $K\RR$, following Atiyah. One can show that the second $C_2$-equivariant Postnikov slice of $K\RR$, denoted $P^2K\RR$ (in the sense of \cite{HHR-kervaire}) has underlying spectrum $KU[0,2]$ and fixed point spectrum $\KO[0,1]$ (\cite{hill-privatecorrespondence}). Therefore both perspectives are encompassed by the space of \textit{equivariant} units of $K\RR$, i.e.~$\GL_1(K\RR)$. In other words, given the correct equivariant generalization of Lemma \ref{lem:truncation commutes with GL1}, both the real and complex Brauer groups described above are determined by a delooping of $P^2\GL_1(K\RR)$. This is closely related the operator theoretic perspective on real $K$-theory described in \cite{moutuou-gradedBrauergroups}. We will return to this question in later work.
\end{rmk}

\subsection{Twisting $\Spin$ and $\String$ structures}\label{sec:twistedStrings}

Recall from \cite{redden-stringstructures} that the set of $\String$-structures on a $\Spin$-manifold $M$ is a torsor for $H^3(M;\ZZ)$. Something similar happens in this case, except that we are considering $\String$-structures on a manifold with a fixed orientation (as opposed to a fixed $\Spin$-structure). In this case we have that the set of $\String$-structures on $M$ relative to a fixed orientation is a torsor for $\gbcconn^0(M)\cong H^1(M;\ZZ/2)\ltimes H^3(M;\ZZ)$). Similarly, the $\Spin$-structures on a real manifold $M$ are a torsor for $\gbrconn^0(M)\cong H^1(M;\ZZ/2)\ltimes H^2(M;\ZZ/2)$. 

\begin{defn}
Let $\SpecialO\modmod\String$ denote the $\EE_\infty$-group arising as the fiber of the connective cover $\BString\to\BSO$.
\end{defn}

\begin{prop}\label{prop:fiber of Bstring to BSO}
There is an equivalence of $\EE_\infty$-groups $\SpecialO\modmod\String\simeq \GBCconn$.
\end{prop}

\begin{proof}
Let $F$ be the fiber in $\Sp$ of the connective cover $\bstring\to\bso$.   Therefore $F$ has $\pi_1(F)\cong \ZZ/2$ and $\pi_3(F)\cong \ZZ$ as its only non-trivial homotopy groups. It suffices to show that the $k$-invariant of $F$ is non-trivial. We will show that the $k$-invariant of $\Sigma F$ is non-trivial, which is equivalent. 

Note that there is a fiber sequence $\bstring\to\bso\to\bso[0,4]\simeq \Sigma F$. By Lemma \ref{noinfB}, there is an equivalence of $h$-types \[\BSO[0,4]\simeq B(\BGL_1(KU)[0,3])\simeq K(\ZZ/2,2)\times K(\ZZ,4)\] so there is a natural isomorphism of sets $\bso[0,4]^0(X)\cong H^2(X;\ZZ/2)\oplus H^4(X;\ZZ)$. If the $k$-invariant of $F$ were trivial, this would be an isomorphism of Abelian groups. We will show that is not the case.  

Because of the $h$-type splitting $\BSO[0,4]\simeq K(\mathbb{Z}/2,2)\times K(\ZZ,4)$ described above there is a projection $\BSO[0,4]\to K(\ZZ/2,2)$ and the composite of that projection with the truncation $\BSO\to \BSO[0,4]$ must be non-trivial. Therefore the second Stiefel Whitney class $w_2\colon \BSO\to K(\ZZ/2,2)$ can be factored as $\BSO\to \BSO[0,4]\to K(\ZZ/2,2)$. As a result, the composite $\bso^0(X)\to\bso[0,4]^0(X)\to H^2(X,\ZZ/2)$ must take an oriented vector bundle $V$ on $X$ to $w_2(V)$. 

The composite $p\colon\BSO\to \BSO[0,4]\to K(\ZZ,4)$ determines \textit{some} integral characteristic class (and is therefore some multiple of the first Pontryagin class $p_1$). We argue that it must be either $p_1$ or $-p_1$. First note that by the computations of \cite{eileberg-maclane-EMspacesII} and the K\"unneth formula, $H_4(\BSO[0,4];\ZZ)$ splits as  $H_4(K(\ZZ/2,2);\ZZ)\oplus H_4(K(\ZZ,4);\ZZ)\cong \ZZ/4\oplus\ZZ$. This implies that $H^4(K(\ZZ/2,2)\times K(\ZZ,4);\ZZ)\cong Hom_{\mathrm{Ab}}(\ZZ/4\oplus\ZZ,\ZZ)\cong \ZZ$. Thus the projection $\BSO[0,4]\to K(\ZZ,4)$ gives an isomorphism in $H^4$. The restriction along $\BSO\to\BSO[0,4]$ is surjective on $H^4$ for connectivity reasons, and hence an isomorphism, so the composite $p$ must be a generator of $H^4(\BSO,\ZZ)$, which proves the claim.

Since the natural map $\bso^0(X)\to\bso[0,4]^0(X)$ is a map of Abelian groups, if the isomorphism $\bso[0,4]^0(X)\cong H^2(X;\ZZ/2)\oplus H^4(X;\ZZ)$ were also one of Abelian groups then the map which sends an oriented vector bundle $V$ to $(w_2(V),\pm p_1(V))$ would be a map of Abelian groups. But the Whitney sum formulas that dictate the behaviour of $w_2$ and $p_1$ under the direct sum of bundles make this impossible.
\end{proof}

\begin{cor}\label{cor:twisting string structures}
    Let $X$ be a space with an oriented real vector bundle $\xi\colon X\to \BSO$ which lifts to a string bundle. Then the set of string bundles which lift $\xi$ is a torsor for $\gbcconn^0(X)\simeq \bgl_1(KU[0,2])^0(X)$.
\end{cor}

\begin{proof}
    This follows from applying the limit preserving functor $Map(X,-)$ to the pullback of $h$-types
    \[
    \begin{tikzcd}
        \GBCconn \ar[r]\ar[d] & \BString\ar[d]\\
        \{\ast\}\ar[r,"\xi"] & \BSO
    \end{tikzcd}
    \]
\end{proof}

\begin{rmk}
    Corollary \ref{cor:twisting string structures} implies that if $X$ is a connected and oriented manifold which admits a string structure then those string structures can be twisted by elements of $\bgl_1^0(KU[0,2])(X)$. These are $KU[0,2]$-line bundles on $X$ and therefore, in light of Corollary \ref{cor:picKU03 classifies super 2line bundles}, complex super 2-line bundles.
\end{rmk}

\begin{rmk}
    We suspect it is also true that $\GBC\simeq \fib(\BString\to\BO)$, but we don't have an interesting interpretation of this fact, so we do not investigate it here. 
\end{rmk}

The following proposition can be proven by methods similar to those used in the proof of Proposition \ref{prop:fiber of Bstring to BSO}:

\begin{prop}\label{prop:fiber of BSpin to BO}
    If $\OO\modmod\Spin$ denotes the fiber of the map $\BSpin\to\BO$ then there is an equivalence of $\EE_\infty$-groups $\OO\modmod\Spin\simeq \GBRconn\simeq \BGL_1(\KO[0,1])$. 
\end{prop}

\begin{cor}
    For an $h$-type $X$ with a real vector bundle $\xi\colon X\to \BO$ the set of lifts of $\xi$ to $\BSpin$ is a torsor for $\bgl_1^0(\KO[0,1])$. 
\end{cor}

\begin{rmk}
    Note that, in light of the results of \cite{ergus-thesis,beardsleybialgebras}, Propositions \ref{prop:fiber of Bstring to BSO} and \ref{prop:fiber of BSpin to BO} imply that $M\String\to M\SpecialO$ and $M\Spin\to M\OO$ are Hopf-Galois extensions (or coGalois extensions) in the sense of \cite{rog} with Galois algebras $\sph[BGL_1(\KU[0,2])]$ and $\sph[BGL_1(KO[0,1])$ respectively. In other words, the $\infty$-groups $\BGL_1(\KU[0,2])$ and $\BGL_1(\KO[0,1])$ are functions on the (non-existent) Galois groups of $M\String\to M\SpecialO$ and $M\Spin\to M\OO$. 
\end{rmk}

\subsection{Twisting Cohomology Theories}

There are maps of $h$-types $\GBC\to\Pic(\KU)$ and $\GBC\to\Pic(\KO)$, but we do not know if these are maps of $\EE_\infty$-groups because we do not know if there are $\EE_\infty$ splittings $\Pic(\KU)\simeq \GBC\times \Pic_4^\infty(\KU)$ or $\Pic(\KO)\simeq\GBR\times\Pic_3^\infty(\KO)$. However, there \textit{are} $\EE_\infty$ splittings $\BGL_1(\KU)\simeq \GBCconn\times \BGL_1(\KU)[4,\infty)$ and $\BGL_1(\KO)\simeq \GBRconn\times\BGL_1(\KO)[3,\infty)$ (this is well known, but also follows from our Corollaries \ref{cor:gl1KUsplitting} and \ref{cor:gl1KOsplitting}). Therefore for a connected $h$-type $X$ there are twists of real and complex $K$-theory by $\gbcconn^0(X)$ and $\gbrconn^0(X)$. 

\begin{ques}
Given a class $\alpha\in\gbcconn^0(X)$ there is a twisted $K$-theory group $K^\alpha(X)$. On the other hand, via the isomorphism of Corollary \ref{cor:GBrU is the same as gbc}, $\alpha$ is also a class in $\mathrm{GBr}\mathbf{U}(X)$, i.e.~a Morita class of graded continuous class $C^\ast$-algebras with spectrum $X$. Then, by \cite[Section 4.1]{maycockparker} there is a twisted operator theoretic $KK$-theory group $KK^\alpha(X)$. Are these two groups always isomorphic? This question, for the comparison between \textit{ungraded} $C^\ast$-algebras and $H^3(X;\ZZ)$ is answered in the affirmative in \cite{hebestreitSagaveTwisted}. 
\end{ques}

\noindent More generally, using the results of Section \ref{sec:twistedStrings}, we may repeat the the constructions of \cite{abg-twists} to obtain twists of other cohomology theories that now can be interpreted as coming from $\KU[0,2]$ and $\KO[0,1]$ line bundles.

\begin{exam}\label{example: spin twists of KO}
    By taking Thom spectra of the fiber sequence $\BGL_1(KO[0,1])\to \BSpin\to BO$ of Corollary \ref{prop:fiber of BSpin to BO} and composing with the $\EE_\infty$ Atiyah-Bott-Shapiro orientation of \cite{joachim-highercoherences}, we obtain a composite of maps of $\EE_\infty$-ring spectra 
    \[
    \Sigma^\infty_+ \BGL_1(\KO[0,1])\to M\Spin\to \KO.
    \] The $\gl_1\vdash \Sigma_+^\infty\Omega^\infty$ adjunction then induces a map of $\EE_\infty$-groups 
    $\BGL_1(\KO[0,1])\to \GL_1(KO)$
    which deloops to \[
    \mathrm{B}^2\GL_1(\KO[0,1])\to \BGL_1(\KO).
    \]
    Recall that $\mathrm{B}^2\GL_1(\KO[0,1])$ is the base space component of $\Br(\KO[0,1])$ and $\BGL_1(\KO[0,1])\simeq\Omega\BGL_1(\KO[0,1])$ classifies super 2-line bundles. We think it reasonable to interpret this map as giving twists of $\KO$-theory by real super 3-line bundles (though there does not appear to be an agreed upon definition of 3-line bundles in the literature). 
\end{exam}

\begin{rmk}
    Note that the fiber of the composite $\BSpin^c\to \BSO\to \BO$ is, as an $h$-type, equivalent to $\ZZ/2\times K(\ZZ,2)$. Under the assumption that this fiber has non-trivial $\EE_\infty$-structure, Theorem \ref{thm:secondkinvariantofpicKO} implies that it is equivalent as an $\EE_\infty$-group to the $h$-type classifying complex superline bundles. In other words, $\Spin^c$ structures on manifolds can be twisted by complex superline bundles.
\end{rmk}

\begin{exam}\label{exam:string twists of tmf}
    Similarly to Example \ref{example: spin twists of KO}, we can use the fiber sequence of Corollary \ref{prop:fiber of Bstring to BSO}, $\BGL_1(\KU[0,2])\to \BString\to\BSO$, along with the $\EE_\infty$-orientation $M\String\to\KU$ of Ando, Hopkins and Rezk \cite{AHR-orientation} to obtain twists \[\mathrm{B}^2\GL_1(\KU[0,2])\to \BGL_1(tmf).\]

    Again one might interpret such twists as twists of $tmf$-theory by complex super 3-line bundles, or maps to the connected component of $\Br(\KU[0,2])$.
\end{exam}

\begin{rmk}\label{rmk:twisting tmf by K(KU)}
Recall that, for a commutative ring spectrum $R$, the base point component of $\Br(R)$ is the space of $\LMod_{R}$-modules (in $\Cat_\infty$) which are equivalent to $\LMod_{R}$ and equivalences between them. Moreover, there is a canonical map of spectra $\bgl_1(R)\to \gl_1(K(R))$ hence a map of $\EE_\infty$-groups $\mathrm{B}^2GL_1(R)\to \BGL_1(K(R))$ by which we may think of maps $X\to \mathrm{B}^2\GL_1(R)$ as $K(R)$-line bundles. Example \ref{exam:string twists of tmf} then suggests that our twists $\mathrm{B}^2\GL_1(\KU[0,2])\to\BGL_1(tmf)$ are related to $K(KU)$ being a form of elliptic cohomology (see for instance \cite{BDRR-stablebundles}). 
\end{rmk}

\subsection{Mathematical Physics}\label{sec:freedphys}

In this section we describe how our work is connected to work in mathematical physics of Freed, Hopkins and others. In \cite[1.34, 1.38]{freedOrientifolds}, Freed describes four spectra: $cAlg_{\RR}^\times$, $cAlg_{\CC}^\times$, $Alg_{\RR}^\times$ and $Alg_{\CC}^\times$. These are each Picard spectra of Morita 2-categories of invertible algebras, bimodules between them, and intertwiners between bimodules. In the first two  cases Freed requires that the bimodule structures and the intertwiners are all continuous with respect the topologies of $\RR$ and $\CC$ respectively. In the second two cases, everything is with respect to the discrete topologies on $\RR$ and $\CC$. Freed computes the homotopy groups and $k$-invariants of each of these spectra (implicitly using results which are made concrete in \cite{johnsonosornoPicard, GJOS-stablePostnikov}). Each of these have a finite number of non-zero homotopy groups, all of which we exhibit below:
\begin{align*}
    \pi_{\{0,1,2,3\}}cAlg_{\CC}^\times &\cong \{\ZZ/2,~\ZZ/2,~0,~\ZZ\}\\
    \pi_{\{0,1,2\}}cAlg_{\RR}^\times &\cong \{\ZZ/8,~\ZZ/2,~\ZZ/2\}\\
    \pi_{\{0,1,2\}}Alg_{\CC}^\times&\cong \{\ZZ/2,~\ZZ/2,\CC^\times\}\\
    \pi_{\{0,1,2\}}Alg_{\RR}^\times &\cong \{\ZZ/8,~\ZZ/2,~\RR^\times\}
\end{align*}
Where, in the last two cases, $\CC^\times$ and $\RR^\times$ both have the discrete topology. Freed also computes the $k$-invariants of these spectra to all be non-trivial. Freed's computations when combined with ours (as well as Conjecture \ref{rmk:k invariants of K of discrete fields}) yield the following:

\begin{thm}
There are equivalences of spectra:
\begin{align*}
    cAlg_{\CC}^\times&\simeq \gbc\\
    cAlg_{\RR}^\times&\simeq \gbr\simeq Alg_{\RR}^\times\\
    Alg_{\CC}^\times[1,2]&\simeq \bgl_1(K(\CC)[0,2])\footnotemark\\
\end{align*}
\end{thm}

\footnotetext{While the non-triviality of the $k$-invariants of this spectrum relies on Conjecture \ref{rmk:k invariants of K of discrete fields}, it does have the correct homotopy groups in the correct degrees (cf.~\cite{Karoubi-BottPeriodicity}).}

In \cite[4.3, 4.4]{freed-anomalies} Freed again introduces these spectra, though with different names (and of course ignoring the difference between $cAlg_{\RR}^\times$ and $Alg_{\RR}^\times$). Specifically, he writes $R^{-1}$ for $cAlg_{\CC}^\times$, $R^{-2}_{\RR/\ZZ}$ for $Alg_{\CC}^\times$, and $E$ for $cAlg_{\RR}^\times$. However, Freed is also interested in the connective covers of the one-fold desuspensions of these spectra, which he denotes by $R^{-2}$, $R^{-3}_{\RR/\ZZ}$ and $E^{-1}$. This immediately implies the following.

\begin{cor}
There are equivalences of spectra:
\begin{align*}
    R^{-2}&\simeq \gl_1(\KU[0,2])\\
    R^{-3}_{\RR/\ZZ}&\simeq \gl_1(K(\CC)[0,1])\footnotemark\\
    E^{-1}&\simeq \gl_1(\KO[0,1])
\end{align*}
\end{cor}
\footnotetext{Again, conjecturally.}

The spectra $R^{-2}$, $R^{-3}_{\RR/\ZZ}$ and $E^{-1}$ are explained by Freed to be the Picard spectra of the groupoids of complex superlines, \textit{flat} complex superlines, and real superlines. This should not be surprising in light of the identifications made in Corollaries \ref{cor:gl1KUsplitting} and \ref{cor:gl1KOsplitting}.

In \cite{freed-anomalies, freedOrientifolds}, Freed identifies $cAlg_{\CC}^\times$ with $\Sigma^{-1}\ko[1,4]$, which is abstractly equivalent to $\gbc$, but remarks that he does not have a conceptual reason for this identification. Similarly, for $cAlg^\times_{\RR}\simeq Alg_{\RR}^\times$, Freed mentions that there is not an ``off the shelf'' spectrum representing it. We claim that the constructions of this paper, and the identifications made in this section, provide spectra which naturally arise in stable homotopy theory (they are ``off the shelf'') which are the Picard spectra of these Morita categories. Moreover, the constructions of Section \ref{sec:SegalGroupoid} and Section \ref{sec:gradedlinebundles} essentially prove that the Picard and unit spectra of interest in this paper have the desired geometric interpretations.

\subsection{The Anderson Dual of $\sph$}\label{sec:anderson}

Recall that there is a spectrum called the \textit{Anderson dual of $\sph$}, denoted $I_\ZZ$, which defines a functor on $\Sp$, $Map(-,I_\ZZ)\colon\Sp^{op}\to\Sp$ ((see e.g.~\cite[\S 2]{heard-stojanoska--realityduality} or \cite[4.3.9]{dag14}). Given a spectrum $E$ we will write $I_\ZZ E$ for $Map(E,I_\ZZ)$ and refer to this as the \textit{Anderson dual of $E$}.

We now show that there is a close relationship between the truncated Picard spectrum $\gbc$ and the Anderson dual of the sphere spectrum. This is consistent with the results of the prior section as the Anderson dual often appears in the mathematical physics work of Freed, Hopkins and others (see for instance \cite[Hypothesis 5.17]{Freed-SRE} and \cite[Ansatz 5.26, Theorem 5.27]{FreedHopkins-reflectionpositivity}).

We will be particularly interested in $\Sigma^n(I_{\ZZ}[-n,\infty))$, the $n$-fold suspension of the $-n$-connective cover of $I_\ZZ$. We will simplify notation by writing this as ${}_nI_\ZZ$. 

\begin{lem}
    The spectrum ${}_3I_\ZZ$ has homotopy groups $\pi_0=\ZZ/2$, $\pi_1=\ZZ/2$, $\pi_2=0$, $\pi_3=\ZZ$. 
\end{lem}

\begin{proof}
The uppermost homotopy groups of $I_\ZZ$ (which is coconnective) are well known, see again \cite[4.3.9]{dag14}, and the result follows by suspending. 
\end{proof}

\begin{lem}
    The unique $k$-invariant of the spectrum ${}_2I_\ZZ$ is $\beta\Sq^2$ and the bottom $k$-invariant of ${}_3I_{\ZZ}$ is $\Sq^2$. 
\end{lem}

\begin{proof}
By \cite[4.2.7(1)]{dag14} the functor $Map(-,I_\ZZ)$ is a contravariant equivalence on the full subcategory of spectra with finitely many homotopy groups all of which are finitely generated. The first $k$-invariant of $\sph$ is non-trivial and therefore so is the map $\Sq^2\circ\rho\colon H\ZZ\to\Sigma^2H\ZZ/2$. This is also of course the $k$-invariant of $\sph[0,1]$. It follows that the uppermost $k$-invariant of $I_\ZZ$, hence the $k$-invariant of $I_\ZZ[-2,0]$, or equivalently of ${}_2I_\ZZ$, is a non-trivial map $H\ZZ/2\to\Sigma^3H\ZZ$, and therefore must be $\beta\Sq^2$, the generator of $H^3(H\ZZ/2;\ZZ)\cong\ZZ/2$. By a similar argument the bottom $k$-invariant of ${}_3I_\ZZ$ is non-trivial and therefore $\Sq^2$. 
\end{proof}

\begin{thm}\label{IZisPicKU}
There is an equivalence of spectra $\gbc\simeq {}_3I_\ZZ$.
\end{thm}

\begin{proof}
The spectra have the same homotopy groups and both have $\Sq^2$ as their bottom $k$-invariant. The calculations of Lemmas \ref{lem:SESforH4pic01} and \ref{lem:Z4computationforKU} and Propositions \ref{prop:KU k invariant generates Z4} and \ref{thm:both k invariants work for picKU} proceed identically for ${}_3I_{\ZZ}$ and show that although it has two possible second $k$-invariants they yield equivalent spectra, both of which must be equivalent to $\gbc$. 
\end{proof}

\begin{cor}
    There is an equivalence of spectra $\gl_1(\KU[0,2])\simeq {}_2I_\ZZ$. 
\end{cor}

We recall the hypothetical characterization of topological field theories given by Freed and Hopkins \cite[Ansatz 5.26]{freedOrientifolds}:

\begin{ansatz}\label{ansatz:freedHopkinsTFTs}
A continuous, invertible, $n$-dimensional extended topological field theory with symmetry group $H_n$ is a map \[\phi\colon \Sigma^nMTH_n\to\Sigma^{n+1}(I_\ZZ[-n,0])\] where $MTH_n$ is the Madsen-Tillman spectrum for $H_n$ of \cite{GTMW-cobordismcategory}.
\end{ansatz}

If we believe Ansatz \ref{ansatz:freedHopkinsTFTs} then a topological field theory determines a map $\phi^\ast\colon MTH_n^n(X)\to {}_nI_\ZZ^0(X)$ for any manifold $X$. The domain of this map is, by \cite{GTMW-cobordismcategory}, a set of submersions $E\to X$ with $n$-dimensional fibers, up to cobordism, and therefore essentially a bundle of cobordism classes of $n$-dimensional manifolds on $X$. If we let $n=1,2$ then Theorem \ref{IZisPicKU} and Corollary \ref{cor:picKU03 classifies super 2line bundles} imply that we have a map whose input is bundles of 1 or 2 dimensional cobordisms over $X$ and whose output is either super lines bundles on $X$ or super 2-line bundles on $X$, which is the behavior one would expect of a fully extended, invertible topological field theory (especially in the case that $X=\ast$). This also suggests that super $n$-lines in general should be classified by maps into ${}_nI_\ZZ$ (or rather, should be defined as such, as the authors are not aware of any generally accepted definition of super $n$-line) . 

The following conjecture is a 2-local and real version of Theorem \ref{IZisPicKU}, but the indeterminacy of the $k$-invariants of $\gbr$ in our calculations prevents us from proving it. Recall that $\pi_{-4}(I_\ZZ)\cong \ZZ/24$ which becomes $\ZZ/8$ after 2-localizing. Therefore the left hand spectrum below has homotopy groups $\{\ZZ/8,~\ZZ/2,~\ZZ/2,~0,~\ZZ_{(2)}\}$, which agree with the homotopy groups of the right hand side.

\begin{conj}
    There is an equivalence of spectra ${}_4(I_\ZZ)_{(2)}\simeq\pic_0^4(\KO_{(2)})$.
\end{conj}

\noindent The fact that $\pi_{-4}(I_\ZZ)\cong\ZZ/24$, combined with Theorem \ref{IZisPicKU} also suggests the following:

\begin{conj}
There is an equivalence $\br(\KU)[0,4]\simeq{}_4I_{\ZZ}$ where $\br(\KU)$ is the Brauer spectrum of $\KU$. In particular, $\pi_0(\br(\KU))\cong\ZZ/24$.
\end{conj}

Recall that $\Pic(\KU)$ is the space of invertible $\KU$-modules and that $\Br(\KU)$ (the underlying space of the Brauer spectrum $\br(\KU)$) is the space of invertible $\LMod_{\KU}$-modules (where we think of $\LMod_{\KU}$ as a commutative algebra object in the $(\infty,2)$-category of presentable $\infty$-categories). Moreover, $\Omega\Pic(\KU)\simeq \GL_1(\KU)$ and $\Omega\Br(\KU)\simeq\Pic(\KU)$. By thinking of modules over $\LMod_{\KU}$, say $\LMod_{\LMod_{\KU}}$, as a commutative algebra in the $(\infty,3)$-category of $(\infty,2)$-categories, we can construct an infinite loop space of invertible $\LMod_{\LMod_{\KU}}$-modules, which we might suggestively write as $\mathbb{G}_m^3(\KU)$. This pattern then of course continues ad infinitum, giving an infinite sequence of $\EE_\infty$-groups $\{\mathbb{G}_m^n(\KU)\}_{n=0}^\infty$ with $\mathbb{G}_m^0(\KU)=\GL_1(\KU)$, $\mathbb{G}_m^1(\KU)=\Pic(\KU)$, etc., and with $\Omega\mathbb{G}_m^n(\KU)\simeq\mathbb{G}_m^{n-1}(\KU)$. This provides the necessary notation for the following conjecture:

\begin{conj}\label{conj:GmConjecture}
For all $n\geq 0$ there are equivalences $\mathbb{G}^n_m(\KU)[0,n+2]\simeq{}_{n+2}I_\ZZ$. Equivalently, there are equivalences $I_{\ZZ}(\mathbb{G}_m^n(KU)[0,n+2])\simeq \sph[0,n+1]$.
\end{conj}

\begin{rmk}
    There is at least one good reason to doubt Conjecture \ref{conj:GmConjecture}. If it were true then it would follow that the stable homotopy groups of spheres (up to Anderson duality) could be recovered from the sequence $\{\mathbb{G}_m^n(KU)\}_{n\geq k}$ for any $k\geq 0$. A mild rebuke to this concern might be the notion of \textit{redshift} in chromatic homotopy theory (cf.~\cite{rognes-redshift}). This is the idea that taking algebraic $K$-theory ``increases chromatic height.'' Recall that $KU$ is of chromatic height 1 (although this is somewhat ambiguous as chromatic homotopy theory takes place after localizing at a prime, but we ignore that for the purposes of speculation). It is known that if one ``understood'' all chromatic heights, then one would know the stable homotopy groups of spheres. Moreover, there is a canonical map of spectra $\pic(\KU)\to \gl_1(K(\KU))$ (although this is not an equivalence, even in low degrees), referenced also in Remark \ref{rmk:twisting tmf by K(KU)}. Similarly, there is a canonical map $\mathbb{G}_m^n(\KU)\to K^n(\KU)$. Therefore one might be forgiven for believing that, at least in low degrees, these two spectra contain the same information and that $\mathbb{G}_m^n(\KU)$ might know a great deal about $\pi_\ast(\sph)$. 
\end{rmk}

We attempt to collect the the ideas of Sections \ref{sec:freedphys} and \ref{sec:anderson}, at least in the complex case, in the below table:

\begin{center}
\begin{tabularx}{\textwidth} { 
  | >{\centering\arraybackslash}X 
  | >{\centering\arraybackslash}X 
  | >{\centering\arraybackslash}X | }
\hline
$\gl_1(\KU)[0,2]\simeq \gl_1(\KU[0,2])$ & $\gbc$ & $\br(KU)[0,4]$\\
 \hline
N/A & central simple $\CC$-superalgebras & central simple $\KU[0,2]$-algebras? \\ 
 \hline
  invertible (super) $\mathbb{C}$-modules & invertible $sVect_{\CC}$-modules ($\cong$ invertible $\KU[0,2]$-modules?) & invertible $\LMod(\KU[0,2])$-modules? \\ 
 \hline
complex superlines & complex super 2-lines & complex super 3-lines? \\ 
\hline
$\Sigma^2(I_{\ZZ}[-2,\infty))$ & $\Sigma^3(I_{\ZZ}[-3,\infty))$ & $\Sigma^4(I_{\ZZ}[-4,\infty))$?\\
\hline
\end{tabularx}
\end{center}
Note that the passage from the first row to the second is obtained by applying the functor $\LMod(-)$ and that the third row is essentially just a renaming of the second row. 

\bibliography{references}
\end{document}